\documentclass[11pt,xcolor=pdftex,a4paper,oneside,centertags,reqno,endnotes]{amsart}

\usepackage{apptools}
\usepackage{chngcntr}

\usepackage{endnotes}
\usepackage{latexsym}
\usepackage{amsmath}
\usepackage{amsfonts}
\usepackage{amssymb}   
\usepackage{epsfig}
\usepackage{amscd}
\usepackage{amsthm}
\usepackage{enumerate}
\usepackage{eucal}
\usepackage[utf8x]{inputenc}
\usepackage{graphics}
\usepackage{caption}
\usepackage[pagebackref]{hyperref}
\usepackage[backrefs,non-compressed-cites,initials,nobysame]{amsrefs}

\usepackage{pdfsync}
\usepackage{color}
\usepackage{mathrsfs}

\usepackage[margin=1.15in]{geometry}

\usepackage[T1]{fontenc}
\usepackage{mathrsfs}
\usepackage{epsfig}
\usepackage{xypic}
\usepackage{verbatim}
\usepackage{multicol}
\usepackage{wrapfig}
\usepackage{pgf,tikz}
\usepackage{mathrsfs}
\usetikzlibrary{arrows}

\numberwithin{equation}{section}

\newtheorem{theorem}{Theorem}[section]
\newtheorem{proposition}[theorem]{Proposition}
\newtheorem{lemma}[theorem]{Lemma}
\newtheorem{corollary}[theorem]{Corollary}
\theoremstyle{definition}
\newtheorem{definition}[theorem]{Definition}
\newtheorem{remark}[theorem]{Remark}

\renewcommand{\Im}{\mathrm{Im}}
\renewcommand{\Re}{\mathrm{Re}}

\newcommand{\Dom}{\mathrm{Dom}}
\newcommand{\Hei}{\mathbb{H}}
\newcommand{\sign}{\mathrm{sign}}

\newcommand{\N}{\mathbb{N}}
\newcommand{\Z}{\mathbb{Z}}
\newcommand{\R}{\mathbb{R}}
\newcommand{\C}{\mathbb{C}}

\usepackage{textcomp}
\usepackage{dsfont}
\usepackage{latexsym}
\usepackage{amssymb}
\usepackage{amsthm}
\usepackage{amsmath}
\DeclareMathAlphabet{\mathpzc}{OT1}{pzc}{m}{en}
\usepackage{yfonts}
\usepackage{xfrac}
\usepackage{fge}

\usepackage{fancyhdr}
\usepackage{multirow}

\usepackage{newlfont}
\usepackage{graphicx}

\usepackage{mathtools}
\usepackage{comment}
\usepackage{indentfirst}
\usepackage{braket}
\usepackage{scalerel}

\DeclarePairedDelimiter{\abs}{\lvert}{\rvert}

{\left\lbrace\begin{array}{@{}l@{}}}%
{\end{array}\right.}

\makeatletter
\newcommand*{\mint}[1]{%
	\mint@l{#1}{}%
}
\newcommand*{\mint@l}[2]{%
	\@ifnextchar\limits{%
		\mint@l{#1}%
	}{%
	\@ifnextchar\nolimits{%
		\mint@l{#1}%
	}{%
	\@ifnextchar\displaylimits{%
		\mint@l{#1}%
	}{%
	\mint@s{#2}{#1}%
}%
}%
}%
}
\newcommand*{\mint@s}[2]{%
	\@ifnextchar_{%
		\mint@sub{#1}{#2}%
	}{%
	\@ifnextchar^{%
		\mint@sup{#1}{#2}%
	}{%
	\mint@{#1}{#2}{}{}%
}%
}%
}
\def\mint@sub#1#2_#3{%
	\@ifnextchar^{%
		\mint@sub@sup{#1}{#2}{#3}%
	}{%
	\mint@{#1}{#2}{#3}{}%
}%
}
\def\mint@sup#1#2^#3{%
	\@ifnextchar_{%
		\mint@sub@sup{#1}{#2}{#3}%
	}{%
	\mint@{#1}{#2}{}{#3}%
}%
}
\def\mint@sub@sup#1#2#3^#4{%
	\mint@{#1}{#2}{#3}{#4}%
}
\def\mint@sup@sub#1#2#3_#4{%
	\mint@{#1}{#2}{#4}{#3}%
}
\newcommand*{\mint@}[4]{%
	\mathop{}%
	\mkern-\thinmuskip
	\mathchoice{%
		\mint@@{#1}{#2}{#3}{#4}%
		\displaystyle\textstyle\scriptstyle
	}{%
	\mint@@{#1}{#2}{#3}{#4}%
	\textstyle\scriptstyle\scriptstyle
}{%
\mint@@{#1}{#2}{#3}{#4}%
\scriptstyle\scriptscriptstyle\scriptscriptstyle
}{%
\mint@@{#1}{#2}{#3}{#4}%
\scriptscriptstyle\scriptscriptstyle\scriptscriptstyle
}%
\mkern-\thinmuskip
\int#1%
\ifx\\#3\\\else_{#3}\fi
\ifx\\#4\\\else^{#4}\fi  
}
\newcommand*{\mint@@}[7]{%
	\begingroup
	\sbox0{$#5\int\m@th$}%
	\sbox2{$#5\int_{}\m@th$}%
	\dimen2=\wd0 %
	\let\mint@limits=#1\relax
	\ifx\mint@limits\relax
	\sbox4{$#5\int_{\kern1sp}^{\kern1sp}\m@th$}%
	\ifdim\wd4>\wd2 %
	\let\mint@limits=\nolimits
	\else
	\let\mint@limits=\limits
	\fi
	\fi
	\ifx\mint@limits\displaylimits
	\ifx#5\displaystyle
	\let\mint@limits=\limits
	\fi
	\fi
	\ifx\mint@limits\limits
	\sbox0{$#7#3\m@th$}%
	\sbox2{$#7#4\m@th$}%
	\ifdim\wd0>\dimen2 %
	\dimen2=\wd0 %
	\fi
	\ifdim\wd2>\dimen2 %
	\dimen2=\wd2 %
	\fi
	\fi
	\rlap{%
		$#5%
		\vcenter{%
			\hbox to\dimen2{%
				\hss
				$#6{#2}\m@th$%
				\hss
			}%
		}%
		$%
	}%
	\endgroup
}

\makeatletter
\def\overbracket#1{\mathop{\vbox{\ialign{##\crcr\noalign{\kern3\p@}
				\downbracketfill\crcr\noalign{\kern3\p@\nointerlineskip}
				$\hfil\displaystyle{#1}\hfil$\crcr}}}\limits}
\def\underbracket#1{\mathop{\vtop{\ialign{##\crcr
				$\hfil\displaystyle{#1}\hfil$\crcr\noalign{\kern3\p@\nointerlineskip}
				\upbracketfill\crcr\noalign{\kern3\p@}}}}\limits}
\def\overparenthesis#1{\mathop{\vbox{\ialign{##\crcr\noalign{\kern3\p@}
				\downparenthfill\crcr\noalign{\kern3\p@\nointerlineskip}
				$\hfil\displaystyle{#1}\hfil$\crcr}}}\limits}
\def\underparenthesis#1{\mathop{\vtop{\ialign{##\crcr
				$\hfil\displaystyle{#1}\hfil$\crcr\noalign{\kern3\p@\nointerlineskip}
				\upparenthfill\crcr\noalign{\kern3\p@}}}}\limits}
\def\downparenthfill{$\m@th\braceld\leaders\vrule\hfill\bracerd$}
\def\upparenthfill{$\m@th\bracelu\leaders\vrule\hfill\braceru$}
\def\upbracketfill{$\m@th\makesm@sh{\llap{\vrule\@height3\p@\@width.7\p@}}%
	\leaders\vrule\@height.7\p@\hfill
	\makesm@sh{\rlap{\vrule\@height3\p@\@width.7\p@}}$}
\def\downbracketfill{$\m@th
	\makesm@sh{\llap{\vrule\@height.7\p@\@depth2.3\p@\@width.7\p@}}%
	\leaders\vrule\@height.7\p@\hfill
	\makesm@sh{\rlap{\vrule\@height.7\p@\@depth2.3\p@\@width.7\p@}}$}
\makeatother

\newcommand{\dd}{\mathrm{d}}

\newcommand{\res}[2]{ \mathrm{Res}\left(#1,#2\right)}

\newcommand{\Ec}{\mathcal{E}}

\newcommand{\Hc}{\mathcal{H}}

\usepackage[normalem]{ulem}

\renewcommand{\phi}{\varphi}

\begin{document}
\title[Asymptotics for the Heat Kernel on H-Type Groups]{Asymptotics for the Heat Kernel on H-Type Groups}
\author[Bruno]{Tommaso Bruno}
\author[Calzi]{Mattia Calzi}

\address{Tommaso Bruno \\Università degli Studi di Genova\\ Dipartimento di Matematica\\ Via Dodecaneso\\ 35 16146 Genova\\ Italy }
\email{brunot@dima.unige.it}

\address{Mattia Calzi \\Scuola Normale Superiore\\ Piazza dei Cavalieri \\ 7 56126 Pisa\\ Italy }
\email{mattia.calzi@sns.it}

\maketitle
\begin{small}
\section*{Abstract}
We give sharp asymptotic estimates at infinity of all radial partial derivatives of the heat kernel on H-type groups. As an application, we give a new proof of the discreteness of the spectrum of some natural sub-Riemannian Ornstein-Uhlenbeck operators on these groups.
\end{small}

\bigskip

\section{Introduction}
\label{sec:asyintro}
Estimates at infinity for the heat kernel on the Heisenberg group or, more generally, H-type groups have attracted a lot of interest in the last decades (see, e.g.,~\cite{Gaveau, HueberMuller, Beals, Eldredge, Li1, Li3}). In the context of H-type groups, in particular, some results were recently  obtained by Eldredge~\cite{Eldredge} and Li~\cite{Li3} independently. In~\cite{Eldredge}, Eldredge provides precise upper and lower bounds for the heat kernel $p_s$ and its horizontal gradient $\nabla_\Hc p_s$. In~\cite{Li3}, Li provides asymptotic estimates for the heat kernel $p_s$, as well as upper bounds for all its derivatives. Nevertheless, to the best of our knowledge, sharp asymptotic estimates at infinity for the derivatives of $p_s$ are still missing. In this paper we address this problem by providing asymptotic expansions at infinity of the heat kernel and of all its derivatives.

Let $G$ be an H-type group identified with $\R^{2 n}\times \R^m$ via the exponential map, and denote by $(x,t)$  its generic element, where $x\in \R^{2n}$ and $t\in \R^m$. It is well known that the heat kernel $p_s$ is a function of $R\coloneqq \abs{x}^2/4$ and $|t|$. Outside the region $\{(x,t)\in G\colon t=0\}$, \emph{any} derivative of $p_s(x,t)$ can {thus} be written as a finite linear combination with smooth coefficients of the functions 
\[p_{s,k_1,k_2}(x,t)=\frac{\partial^{k_1}}{\partial R^{k_1}}\frac{\partial^{k_2}}{\partial |t|^{k_2}} p_{s}(x,t),\] for {suitable} $k_1,k_2 \in \N$.  We call {these functions} \emph{radial partial derivatives} of $p_s$. Thus, everything can be reduced to finding asymptotic estimates at infinity of $p_{s,k_1,k_2}$ for every $k_1,k_2 \in \N$; these will yield asymptotic estimates of every desired derivative of $p_s$.

We divide the paper in five sections. In the next section we fix the notation and recall some preliminary facts on H-type groups and the method of stationary phase. In the central Sections \ref{sec:m1} and \ref{sec:mgeneral} the functions $p_{s,k_1,k_2}$ are studied. In Section \ref{sec:m1} we provide asymptotic estimates for $p_{s,k_1,k_2}$ in the case $m=1$, namely when $G$ is a Heisenberg group; in Section \ref{sec:mgeneral} we extend the results of Section \ref{sec:m1} to the more general class of H-type groups. This is done via a reduction to the case $m=1$ when $m$ is odd; a descent method is then applied in order to cover the case $m$ even. The preliminary study of the case $m=1$ is necessary except in a single case, for which the general case could be treated directly; nevertheless, we include both proofs for the sake of clarity. As the reader may see, our Theorem~\ref{HstimeI} and Corollary~\ref{Hstime234} cover the cases of~\cite[Theorems 1.4 and 1.5]{Li3} and~\cite[Theorem 4.2]{Eldredge} as particular instances, and imply~\cite[Theorems 1.1 and 1.2]{Li3} and~\cite[Theorem 4.4]{Eldredge} as easy corollaries, by means of formula~\eqref{gradient}. In Section~\ref{OU} we show an interesting application of our estimates, providing a different proof of a theorem due to Inglis~\cite{Inglis} which concerns the discreteness of the spectrum of some Ornstein-Uhlenbeck operators on $G$.

We emphasize that our methods are strongly related to those employed by Gaveau~\cite{Gaveau} and then Hueber and M\"uller~\cite{HueberMuller} in the case of the Heisenberg group; some ideas are also taken from the work of Eldredge~\cite{Eldredge}. In particular, we borrow from~\cite{Gaveau} and~\cite{HueberMuller} the use of the method of \textit{stationary phase}, though in a stronger form provided by H\"ormander~\cite{Hormander2}.

\section{Preliminaries}\label{sec:prel}
\subsection{H-type Groups} 
An H-type group $G$ is a 2-step stratified group whose Lie algebra $\mathfrak{g}$ is endowed with an inner product $(\, \cdot\,,\, \cdot \,)$ such that
\begin{itemize}
	\item[1.] if $\mathfrak{z}$ is the centre of $\mathfrak{g}$ and $\mathfrak{h}=\mathfrak{z}^\perp$, then $[\mathfrak{h},\mathfrak{h}]=\mathfrak{z}$;
	\item[2.] for every $Z \in \mathfrak{z}$, the map $J_Z\colon \mathfrak{h} \to \mathfrak{h}$, 
	\[(J_Z X,Y)=( Z, [X,Y]) \qquad \forall X,Y\in \mathfrak{h},\]
	is an isometry whenever $(Z, Z)=1$.
\end{itemize}
In particular, $\mathfrak{g}$ stratifies as $\mathfrak{h} \oplus \mathfrak{z}$. It is very convenient, however, to realize an H-type group $G$ as $\R^{2 n}\times \R^m $, for some $n,m\in \N$, via the exponential map. More precisely, we shall denote by $(x,t)$ the elements of $G$, where  $x\in \R^{2n}$ and $t \in \R^m$. We denote by  $(e_1,\dots,e_{2n})$ and $(u_1,\dots,u_m)$ the standard bases of $\R^{2 n}$ and $\R^{m}$ respectively. Under this identification, the Haar measure $d y$ is the Lebesgue measure. The maps $\{J_Z \colon Z\in \mathfrak{z}\}$ are identified with $2n \times 2n$ skew symmetric matrices $\{J_t \colon t\in \R^m\}$ which are orthogonal whenever $|t|=1$. This identification endows $\R^{2n}\times \R^{m}$ with the group law
\[(x,t)\cdot (x',t') = \left(x+x',t+t' + \frac{1}{2} \sum_{k=1}^m (J_{u_k}x,x') u_k\right).\]
A basis of left-invariant vector fields for $\mathfrak{g}$ is
\[X_j = \partial_{x_j} + \frac{1}{2}\sum_{k=1}^{m}( J_{u_k} x,e_j) \partial_{t_k},\quad j=1,\dots,2n; \qquad T_k = \partial_{t_k}, \quad k=1,\dots,m.\]
In particular, $(X_j)_{1\leq j \leq 2n}$ is a basis for the first layer $\mathfrak{h}\cong \R^{2 n}$. If $f$ is a sufficiently {smooth} function on $G$, its horizontal gradient will be the vector field $\nabla_\mathcal{H} f \coloneqq \sum_{j=1}^{2n} (X_j f)X_j$, and its sub-Laplacian $\mathcal{L}f\coloneqq -\sum_{j=1}^{2n} X_j^2f$. We refer the reader to~\cite{Bonfiglioli} for further details.

\subsection{The Heat Kernel} On an H-type group $G \simeq \R^{2 n}\times \R^m $ the heat kernel $(p_s)_{s>0}$ has the form
\begin{equation}
	p_s(x,t)= \frac{1}{(4\pi)^n (2\pi)^m s^{n+m}}\int_{\R^m}e^{\frac{i}{s}(\lambda, t)-\frac{\abs*{x}^2}{4s}|\lambda|\coth(|\lambda|)} \left(\frac{ |\lambda|}{\sinh|\lambda|}\right)^n\,\dd\lambda,
	\label{HeatKernel}
\end{equation}
for every $s>0$ and every $(x,t)\in G$ (see~\cite{Gaveau} or~\cite{Hulanicki} for the Heisenberg groups, \cite{Randall} or \cite{Yang} for H-type groups). {For the sake of clarity, we shall sometimes stress the dependence of $p_s$ on the dimension $m$ of the centre of $G$ by writing $p^{(m)}_s$ instead of $p_s$.} 

We begin by writing the heat kernel \eqref{HeatKernel} in a more convenient form. Let $\mathcal{R}$ be an isometry such that $\mathcal{R} t = |t|u_1$, where $u_1$ is the first element of the canonical basis\footnote{The choice of $u_1$ is  actually irrelevant.} of the centre of $G$, namely $\R^m$. Then make the change of variables $\lambda \mapsto \mathcal{R}^{-1} \lambda$ in \eqref{HeatKernel}, which gives
\begin{equation}\label{HeatKernelMod}
	p_s(x,t)= \frac{1}{(4\pi)^n (2\pi)^m s^{n+m}}\int_{\R^m}e^{\frac{i}{s}( \lambda, u_1) \abs{t}-\frac{\abs{x}^2}{4s}|\lambda|\coth(|\lambda|)} \left(\frac{ |\lambda|}{\sinh|\lambda|}\right)^n\,\dd\lambda.
\end{equation}
It is now more evident that $p_s$ depends only on $\abs{x}$ and $\abs{t}$. This leads us to the following definition.

\begin{definition}\label{defp1k1k2} Let $R=\frac{\abs{x}^2}{4}$. 
	For all $s>0$ and for all $k_1,k_2\in \mathbb{N}$, define
	\begin{equation}\label{pk1k2} 
		\begin{split}
			p_{s,k_1,k_2}(x,t)&\coloneqq\frac{\partial^{k_1}}{\partial R^{k_1}} \frac{\partial^{k_2}}{\partial |t|^{k_2}}p_{s}(x,t)= \frac{ (-1)^{k_1}i^{k_2}}{(4\pi)^{n}(2\pi)^{m} s^{n+m+k_1+k_2}} \\&\quad\quad \times \int_{\mathbb{R}^m} e^{\frac{i}{s}|t|( \lambda, u_1) -\frac{\abs{x}^2}{4 s}|\lambda| \coth|\lambda|} \frac{\abs{\lambda}^{n+k_1}\cosh(|\lambda|)^{k_1}}{\sinh(|\lambda|)^{n+k_1}}(\lambda,u_1)^{k_2}\,\dd \lambda.
		\end{split}
	\end{equation}
\end{definition}

Notice that $p_s$ is a smooth function of $R$ and $\abs{t}$ {by formula~\eqref{HeatKernelMod}}, so that the definition of $p_{s,k_1,k_2}$ is meaningful on the whole of $G$. Moreover, consider a differential operator on $G$ of the form
\[
X^\gamma = \frac{\partial^{\abs{\gamma}}}{\partial x^{\gamma_1} \partial t^{\gamma_2}}
\]
for some $\gamma= (\gamma_1,\gamma_2)\in \N^{2n}\times \N^m$. By means of Faà di Bruno's formula, the function $X^\gamma p_s$ can be written {on $\{t\neq 0\}$} as a finite linear combination with smooth coefficients of the functions $p_{s,k_1,k_2}$, for suitable $k_1$ and $k_2$. Since $X^\gamma p_s$ is uniformly continuous, the value of $X^\gamma p_s(x,0)$ can then be recovered by continuity uniformly in $x\in \R^{2n}$. Therefore, one can obtain asymptotic estimates for $X^\gamma p_s$ by combining appropriately some given estimates of $p_{s,k_1,k_2}$ (see also Remark~\ref{rem:sharpness}). We shall see an application of this in Section~\ref{OU}.

Observe that it will be sufficient to study $p_{1,k_1,k_2}$, since
\[
p_{s,k_1,k_2}(x,t)= \frac{1}{s^{n+m+k_1+k_2}} p_{1,k_1,k_2}\left(\frac{x}{\sqrt{s}},\frac{t}{s}\right)
\]
for every $s>0$, $k_1,\,k_2 \in \N$ and  $(x,t)\in G$. Hence, we shall focus only on $p_{1,k_1,k_2}$.
Moreover, from now on we shall fix the integers $k_1,k_2\geq 0$. {Of course, the choice $k_1=k_2=0$ gives the heat kernel~$p_s$.}

\begin{remark}
	It is well known (see~\cite{Eckmann} or~\cite[Remark 3.6.7]{Bonfiglioli}) that there exist $n$ and $m$ for which $\R^{2n} \times \R^m$ cannot represent any H-type group. Nevertheless,~\eqref{HeatKernel} and hence~\eqref{pk1k2} make sense for every positive $n,m\in \N$, and for such $n$ and $m$ we shall then study $p_{s,k_1,k_2}$.
\end{remark}

\begin{definition}\label{Romegadeltakappa}(cf.~\cite{HueberMuller})
	For every $(x,t) \in G$, define\footnote{Actually, $\omega$ is defined for $x\neq0$ and $\delta$ for $t\neq 0$, but we shall not recall it again in the following.} \[\omega \coloneqq \frac{\abs*{t}}{R}, \qquad \delta \coloneqq \sqrt{\frac{R}{\pi \abs*{t}}},\qquad \kappa \coloneqq  2\sqrt{\pi \abs*{t}R}.\]
\end{definition}
We shall split the asymptotic condition  $(x,t)\to \infty$ into four cases, some of which depend on an arbitrary constant $C>1$. In particular, the first one covers the case $\abs{t}/\abs{x}^2$ bounded, while the other three are a suitable splitting of the case $\abs{t}/\abs{x}^2 \to \infty$.

\begin{multicols}{2}{
	\begin{itemize}
	\setlength{\itemindent}{0.2in}
			\item[\bf{I.}] $(x,t)\to \infty$ and $\omega= 4\abs{t}/\abs{x}^2 \leq$~$C$;
			\smallskip
			\item[\bf{II.}] $\delta \rightarrow 0^+$ and $\kappa \rightarrow +\infty$ ;
			\bigskip	
			\item[\bf{III.}] $\delta\rightarrow 0^+$ and {$\kappa\in \left[1/C, C\right]$};
			\bigskip	
			\item[\bf{IV.}] $\kappa\to 0^+$ and $\abs{t}\to +\infty$.
			\bigskip
		\end{itemize}
		\columnbreak
		
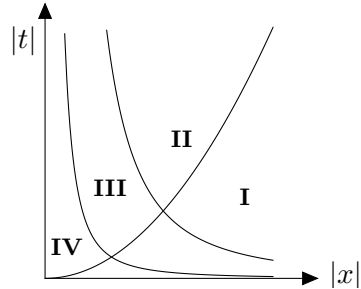
			\begin{tikzpicture}[line cap=round,line join=round,>=triangle 45,x=1.0cm,y=1.0cm, scale=0.60]
			\clip(-2,-0.2) rectangle (7,6.1);
			\draw [->] (0,0) -- (0,6.1);
			\draw [->] (0,0) -- (6,0);
			\draw[smooth,samples=100,domain=0.0:5.0, variable=\t] plot({\t,\t^2/4.5});
			\draw[smooth,samples=100,domain=0.43:5.0, variable=\t] plot({\t},{1/\t^2});
			\draw[smooth,samples=100,domain=1.35:5.0, variable=\t] plot({\t},{10/\t^2});
			\draw (-0.1,1.1) node[anchor=north west] {\footnotesize{\textbf{IV}}};
			\draw (0.84,2.4) node[anchor=north west] {\footnotesize{\textbf{III}}};
			\draw (2.53,3.5) node[anchor=north west] {\footnotesize{\textbf{II}}};
			\draw (4.04,2.22) node[anchor=north west] {\footnotesize{\textbf{I}}};
			\draw (6,0.6) node[anchor=north west] {$\abs{x}$};
			\draw (-1,5.8) node[anchor=north west] {$\abs*{t}$};
			\end{tikzpicture}
	}\end{multicols}
	We shall {describe} the asymptotic {behaviour} of $p_{1,k_1,k_2}$ in each of these four cases. The first two will both need the method of stationary phase (Theorem \ref{stationaryphase} below), while the other two can be treated through Taylor expansions.
	
	In order to simplify the notation, we give some definitions.
	
	\begin{definition}
		Define the function $\theta \colon (-\pi, \pi) \rightarrow \R$ by
		\[\theta(\lambda)\coloneqq \begin{cases}
		\frac{2\lambda-\sin(2\lambda)}{2\sin^2(\lambda)}, 	& \text{if $\lambda \neq 0$,}\\
		0, & \text{if $\lambda= 0$.}\end{cases}\]
		
	\end{definition}
	
	\begin{lemma}\emph{\cite[§ 3, Lemma 3]{Gaveau}} \label{stime:lem:10}
		$\theta$ is an odd, strictly increasing analytic diffeomorphism between $(-\pi,\pi)$ and $\R$.
	\end{lemma}
	
	\begin{definition}	
		For every $\omega \in \R$, set $y_\omega\coloneqq \theta^{-1}(\omega)$. For every $(x,t)\in G$ define
		\[
		d(x,t)\coloneqq \begin{cases}
		\abs{x}\frac{y_\omega}{ \sin(y_\omega)}  &\text{if $x\neq 0$ and $t\neq 0$,}\\
		\abs{x} &\text{if $t=0$,}\\
		\sqrt{4\pi \abs{t}} &\text{if $x=0$.}		
		\end{cases}
		\]
	\end{definition}
	
	It is worth observing that $d(x,t)$ is the Carnot-Carathéodory distance between $(x,t)$ and the origin with respect to the horizontal distribution generated by the vector fields $X_1,\dots, X_{2 n}$. See~\cite[\S 4]{Beals} or~\cite[Theorem 3.5]{Eldredge} for a proof and further details.
	
	\subsection{The Method of Stationary Phase} 
	The main tool that we shall use is an easy corollary of H\"ormander's theorem of stationary phase~\cite[Theorem 7.7.5]{Hormander2}, stated in a form convenient for our needs. We include a proof for the sake of clarity. Given an open set $V\subseteq \R^m$, we write $\mathcal{E}(V)$ for the space of $C^\infty$ complex-valued functions on $V$, endowed with the topology of locally uniform convergence of all the derivatives. If $f$ is a twice differentiable function on an open neighbourhood of $0$, we write $P_{2,0}f$ for the Taylor polynomial of order $2$ about $0$ of $f$.
	
	\begin{theorem}\label{stationaryphase}
		Let $V$ be an open neighbourhood of $0$ in $\R^m$, and let $\mathscr{F},\,\mathscr{G}$ be bounded subsets of $\mathcal{E}(V)$ such that
		\begin{itemize}
			\item[1.] $\Im f(\lambda)\geq 0$ for every $\lambda\in V$ and every $f\in \mathscr{F}$. Moreover, there exist $\eta>0$ and $c_1>0$ such that $B(0,2\eta)\subseteq V$ and $\Im f(\lambda)\geq c_1|\lambda|$ whenever $\abs{\lambda}\geq \eta$ and $f\in \mathscr{F}$;
			
			\item[2.] $\Im f(0)=f'(0)=0$ and $ \det f''(0)\neq 0$ for all $f\in \mathscr{F}$;
			
			\item[3.] there exists $c_2>0$  such that $\abs{f'(\lambda)} \geq c_2\abs{\lambda}$ for all $\abs{\lambda}\leq 2\eta$ and for all $f\in \mathscr{F}$;
			
			\item[4.] there exists $c_3>0$ such that $\abs{g(\lambda)}\leq c_3 e^{c_3 \abs{\lambda}}$ whenever $\lambda\in V$, for every $g\in \mathscr{G}$.
		\end{itemize}
		Then, for every $k\in \N$,
		\begin{equation}\label{eqfasestaz}
			\int_{V} e^{iR f(\lambda)} g(\lambda)\,\dd \lambda= e^{i R f (0)} \sqrt{\frac{(2\pi i)^m}{R^m \det f''(0)}}\sum_{j=0}^k \frac{L_{j,f} g}{R^j}+ O\left( \frac{1}{R^{\frac{m}{2}+k+1}}\right) 
		\end{equation}
		as $R\to +\infty$, uniformly as $f\in \mathscr{F}$ and $g\in \mathscr{G}$, where
		\[
		\begin{split}
		L_{j,f}g= i^{-j}\sum_{\mu=0}^{2j}\frac{ (f''(0)^{-1} \partial, \partial )^{\mu+j} [ (f-P_{2,0}f )^\mu g ](0)}{ 2^{\mu+j} \mu !(\mu+j) !}.
		\end{split}
		\]
		In particular, $L_{0,f}g=g(0)$.
	\end{theorem}
	
	\begin{proof}
		Take some $\tau\in C_c^\infty(\R^m)$ such that $\chi_{B(0,\eta)}\leq \tau \leq \chi_{B(0,2\eta)}$. Then split the integral as 
		\[\int_{V} e^{i R f(\lambda)} g(\lambda)\,\dd \lambda = \int_{V} e^{i R f(\lambda)} g(\lambda)\tau(\lambda)\,\dd \lambda + \int_{V} e^{i R f(\lambda)} g(\lambda)(1-\tau(\lambda))\,\dd \lambda\]
		and apply~\cite[Theorem 7.7.5]{Hormander2} to the first term, thanks to the first assumption in 1 and the assumptions 2 and 3: this represents the main contribution to the integral, and gives the right hand side of~\eqref{eqfasestaz}. The second term is instead negligible, since by the second assumption in 1 and by 4 we get, if $R$ is large enough,
		\[\begin{split}
		\abs*{\int_{V} e^{iR f(\lambda)} g(\lambda)(1-\tau(\lambda))\,\dd \lambda} &\lesssim  \int_{\abs{\lambda}\geq \eta} e^{-R\,\Im f(\lambda)+c_3 \abs{\lambda}}\,\dd \lambda \lesssim \int_\eta^\infty e^{-R c_1\rho+c_3\rho} \rho^{m-1}\,\dd \rho \\
		&=\int_\eta^\infty e^{-(c_1 R \rho- (1+c_3)\rho)-\rho} \rho^{m-1}\,\dd \rho \\
		&\lesssim e^{-(c_1 R-(1+c_3)) \eta}\int_0^\infty e^{-\rho   } \rho^{m-1}\,\dd\rho,\end{split}\]
		which is $O\left( e^{-Rc_1\eta}\right)$. The proof is complete.
	 \end{proof}
	\begin{remark}\label{Laplace}
		Theorem~\ref{stationaryphase} covers more cases than only oscillatory integrals. Indeed, assume we have an integral of the form
		\[\int_{V} e^{-R f(\lambda)} g(\lambda)\,\dd \lambda\]
		where $f$ is real. Under suitable assumptions, such integrals are usually treated via Laplace's method (see, e.g.,~\cite{Erdelyi} and~\cite{Wong}). In this case, one can use directly Theorem~\ref{stationaryphase},
		by substituting $\Im f$ by $f$ in the assumptions 1-4, thus getting
		\begin{equation}\label{eqfasestazmodificata}
			\int_{V} e^{-R f(\lambda)} g(\lambda)\,\dd\lambda= \sqrt{\frac{(2\pi)^m}{R^m \det f''(0)}}\sum_{j=0}^k \frac{L_{j,f} g}{R^j}+ O\left( \frac{1}{R^{\frac{m}{2}+k+1}}\right) ,
		\end{equation}
		with the obvious modifications on $L_{j,f}g$. Coherently, in such cases Theorem~\ref{stationaryphase} will be referred to as \emph{Laplace's method}.
	\end{remark}
	
	\section{The Heisenberg Group}\label{sec:m1}
	In this section we deal with the case $m=1$, namely when $G= \Hei^n$ is the Heisenberg group. The function $p_{1,k_1,k_2}$ of Definition~\ref{defp1k1k2} here reads 
	\[ \begin{split}
	p_{1,k_1,k_2}(x,t)&=\frac{2 (-1)^{k_1}i^{k_2}}{(4\pi)^{n+1} }\int_\mathbb{R} e^{i\lambda \abs{t}-\frac{\abs{x}^2}{4 }\lambda \coth(\lambda)} \frac{\lambda^{n+k_1+k_2}\cosh(\lambda)^{k_1}}{\sinh(\lambda)^{n+k_1}}\,\dd \lambda.
	\end{split}
	\]
	Indeed, the absolute values of $\lambda$ in the integral~\eqref{pk1k2} can be removed by parity reasons. We begin by introducing some functions which greatly simplify the notation.
	\begin{definition} 
		Define
		\begin{equation*}\label{hk1k2}
			h_{k_1,k_2}(R,t)\coloneqq (-1)^{k_1}i^{k_2}\int_\mathbb{R} e^{i\lambda \abs{t}- R\lambda \coth(\lambda)} \frac{\lambda  ^{n+k_1+k_2 }\cosh(\lambda  )^{k_1}}{\sinh(\lambda  )^{n+k_1}}\,\dd  \lambda = \int_\mathbb{R} e^{iR\phi_\omega(\lambda)}a_{k_1,k_2}(\lambda)\,\dd \lambda,
		\end{equation*}
		where
		\begin{equation}\label{ak1k2}
			\begin{split}
				a_{k_1,k_2}(\lambda)&=\begin{cases}
				(-1)^{k_1}i^{k_2}\frac{\lambda  ^{n+k_1 +k_2}\cosh(\lambda  )^{k_1}}{\sinh(\lambda  )^{n+k_1}} & \text{if $\lambda  \not\in \pi i\mathbb{Z}$,}\\
				(-1)^{k_1}i^{k_2}\delta_{k_2,0}	&\text{if $\lambda  =0$,}
				\end{cases}  	\\
				\phi_\omega  (\lambda)&= \begin{cases}
					\omega \lambda   +i\lambda  \coth(\lambda  ) & \text{if $\lambda  \not\in \pi i\mathbb{Z}$,}\\
					i	&\text{if $\lambda  =0$}.
				\end{cases}
			\end{split} 
		\end{equation}
	\end{definition}
	Notice that
	\[
	p_{1,k_1,k_2}(x,t)= \frac{2}{(4\pi)^{n+1}} h_{k_1,k_2}\left(R,t\right)
	\]
	for all $(x,t)\in \mathbb{H}^n$; hence we can reduce matters to studying $h_{k_1,k_2}(R,t)$. Observe moreover that $y_\omega= \theta^{-1}(\omega)  \in [0,\pi)$, since $\omega\geq 0$.
	
	It will be convenient to reverse the dependence relation between $(R,\omega)$ and $(x,t)$: hence, we shall no longer consider $R$ and $\omega$ as functions of $(x,t)$, but rather as ``independent variables''. In this order of ideas, the formula $\abs{t}=R\,\omega$ should sound as a definition.
	
Our intent will be to apply Theorem~\ref{stationaryphase} to a function closely related to $h_{k_1,k_2}$; hence we shall find some stationary points of the phase of $h_{k_1,k_2}$, namely $\phi_\omega$. The lemma below is of fundamental importance.

	\begin{lemma}\emph{\cite[§ 3, Lemma 6]{Gaveau}}\label{stime:lem:7}
		$\phi_\omega'(\lambda )= \omega + \tilde\theta( i\lambda  )$
		for all $\lambda\not\in \pi i\Z^*$, where $\tilde\theta$ is the analytic continuation of $\theta$ to $\Dom(\phi_{\omega})$. In particular, $i y_\omega$ is a stationary point of $\phi_\omega$.
	\end{lemma}
	
	\subsection{\textbf{I.} Estimates for $(x,t)\to \infty$ while $ 4\abs{t}/\abs{x}^2 \leq C$.}\label{sec_HeisI}
		\begin{theorem}\label{HeisHstimeI}
Fix $C>0$. If $(x,t)\to \infty$ while $0\leq \omega \leq C$, then
\begin{equation}\label{eqHeisHstime}
p_{1,k_1,k_2}(x,t)= \frac{1}{|x|}e^{-\frac{1}{4}d(x,t)^2}\Psi(\omega) \left[(-1)^{k_1+k_2}  \frac{ y_\omega^{n+k_1+k_2}\cos(y_\omega)^{k_1}}{\sin (y_\omega)^{n+k_1}} +O\left(\frac{1}{|x|^2}\right)\right]
\end{equation}
where
\[ \Psi(\omega)=\begin{cases}
	\frac{1}{4^n \pi^{n+1}}\sqrt{\frac{\pi \sin(y_\omega)^3}{\sin(y_\omega)-y_\omega\cos(y_\omega)}}, 	& \text{if $\omega\neq 0$,}\\
		\frac{(3\pi)^{1/2}}{4^n \pi^{n+1}}, & \text{if $\omega= 0$.}\end{cases}\]
		\end{theorem}
It is worthwhile to stress that the above estimates may \emph{not} be sharp when $\omega\to 0$ and $k_2>0$, as well as when $\omega\to \frac{\pi}{2}$ and $k_1>0$. In these cases indeed $y_\omega \to 0$ and $y_\omega \to \frac{\pi}{2}$, respectively, and the first term of the asymptotic expansion~\eqref{eqHeisHstime} {may be}  smaller than the remainder. 
{However}, the \emph{sharp} asymptotics of $p_{1,k_1,k_2}$ when $\omega$ remains bounded turn out to be {more} involved, and for the moment we avoid to treat the complete picture of its asymptotic behaviour. The statement above is just a simplified version of Theorem~\ref{HstimeI} of Section~\ref{SecStimeIm>1}, where the general case of H-type groups is completely described.

 In this section we {then} limit ourselves to consider Theorem~\ref{HeisHstimeI} in the stated form. Its proof mostly consists in a straightforward generalization of~\cite[Theorem 2 of § 3]{Gaveau}, but it can also be seen as Proposition~\ref{propHorm} of Section~\ref{SecStimeIm>1} in the {current setting} of Heisenberg groups. Nevertheless, for the sake of completeness we give a brief sketch of the proof.

	{The main idea is to change the contour of integration in the integral defining $h_{k_1,k_2}$ in order to meet a stationary point of $\phi_\omega$. Since $\Im \,\phi_\omega (\lambda)= \omega\, \Im\, \lambda+\Re\left[ \lambda\coth(\lambda)\right]$ for every $\lambda \not\in{\pi i\Z}$, to make this change we need to deepen our knowledge of $\Re\left[ \lambda\coth(\lambda)\right]$ and $\abs{a_{k_1,k_2}}$; this is done in the following lemma, which we state without proof. }
	
	\begin{lemma}\label{stime:lem:5}
		For all $\lambda,y\in \R$ such that $\abs*{\lambda}>\abs*{y}$,
		\[
		\Re[ (\lambda+i y)\coth(\lambda+i y)]=\frac{\lambda\sinh(2\lambda)+y\sin(2y)}{2(\sinh(\lambda)^2+\sin(y)^2)}>0.
		\]
		Moreover, for all $\lambda,y\in \R$ such that either $y\not\in \pi\Z$ or $\lambda\neq 0$,
		\[
		\abs*{a_{k_1,k_2}(\lambda+i y)}= \frac{\abs*{\lambda+i y}^{n+k_1+k_2}(\sinh(\lambda)^2+\cos(y)^2)^{\frac{k_1}{2}}}{(\sinh(\lambda)^2+\sin(y)^2)^{\frac{n+k_1}{2}}}.
		\]
	\end{lemma}
	
	In the following lemma we perform the change of the contour of integration in the definition of $h_{k_1,k_2}$. Its proof is a simple adaptation of that of~\cite[Lemma 1.4]{HueberMuller}.
	
	\begin{lemma}\label{stime:lem:8}
		For all $y\in [0,+\infty)\setminus {\pi \N^*}$
		\[
		\begin{split}
		h_{k_1,k_2}(R,t)= \int_{\R} e^{i R\phi_\omega(\lambda+ i y)} a_{k_1,k_2}(\lambda+i y)\,\dd  \lambda+2\pi i \sum_{\substack{k\in\N^* \\k\pi \in [0,y]}} \res{e^{i R\phi_\omega} a_{k_1,k_2}}{k\pi i }.
		\end{split}
		\]
	\end{lemma}

	\begin{proof}[Proof of Theorem~\ref{HeisHstimeI}] 
		Define
		\[
		\psi_\omega\:= \phi_\omega(\,\cdot\,+i y_\omega)-\phi_\omega(i y_\omega)
		\]
		and observe that
		\begin{equation*}
			\begin{split}
				\phi_\omega(i y_\omega)&=	i\omega\, y_\omega+i y_\omega\cot (y_\omega) =i\frac{y_\omega^2}{\sin(y_\omega)^2},
			\end{split}
		\end{equation*}
		since $\omega=\theta(y_\omega)$.	Therefore, by Lemma~\ref{stime:lem:8} (recall that $0\leq y_\omega <\pi$, so that there are no residues)
		\[
		h_{k_1,k_2}(R,{t})=e^{-\frac{1}{4}d(x,t)^2} \int_\R e^{i R\psi_\omega(\lambda)}a_{k_1,k_2}(\lambda+i y_\omega)\,\dd  \lambda.
		\]
		Our intent is to apply Theorem~\ref{stationaryphase} to the bounded subsets $\mathscr{F}=\{\psi_\omega \colon \omega\in \left[0, C\right]\}$ and $\mathscr{G}= \{a_{k_1,k_2}(\cdot \,+i y_\omega)\colon \omega\in \left[0, C\right]\}$ of $\Ec(\R)$. Therefore we first verify that the four conditions of its statement hold.
		\begin{enumerate}
			\item[2.] Lemmata~\ref{stime:lem:7} and~\ref{stime:lem:10} imply that $ i\phi''_\omega\left(i y_\omega\right) 	=-\theta'(-y_\omega)< 0$ for all $\omega\in \R_+$. From the definition of $\psi_\omega$ we then get
			\begin{equation}\label{dimStimeIcaso2}
				\psi_\omega(0)=\psi'_\omega(0)=0, \qquad i\psi_\omega''(0)< 0.
			\end{equation}
			
			\item[3.]  Consider the mapping  $\psi\colon \R\times (-\pi,\pi)\ni(\lambda,y)\mapsto \psi_{\theta(y)}(\lambda)$. By~\eqref{dimStimeIcaso2}, $\partial_1\psi(0,y)=0$ and $i\partial_1^2\psi(0,y)< 0$ for all $y\in [0,\pi)$; moreover, $\psi$ is analytic thanks to Lemma~\ref{stime:lem:10}. Therefore, by Taylor's formula we may find two constants $\eta >0$ and $C'>0$ such that $ \abs{\partial_1\psi(\lambda,y)}\geq C' \abs*{\lambda}$ for all $\lambda\in [-2\eta,2\eta]$ and for all $y\in [0,\theta^{-1}(C)]$.
			
			\item[1.]Lemma~\ref{stime:lem:5} implies that
			\[
			\Im\,\psi(\lambda,y) =\frac{\lambda\cosh(\lambda)\sinh(\lambda)-y\cot(y)\sinh(\lambda)^2}{\sinh(\lambda)^2+\sin(y)^2}
			\]
			for all $\lambda\in\R$ and for all $y\in (-\pi,\pi)$, $y\neq0$; moreover, the mapping $(0,\pi)\ni y\mapsto y\cot(y)$ is strictly decreasing and tends to $1$ as $y\to 0^+$. Therefore, if $\lambda\neq 0$ and $y\in [0,\pi)$, then
			\[
			\begin{split}
			\Im\,\psi(\lambda,y) \geq \frac{\lambda\coth(\lambda)-1}{1+\frac{1}{\sinh(\lambda)^2}}>0
			\end{split}
			\]
			since $\lambda\coth(\lambda)-1>0$.
			Observe finally that, since  $\frac{\lambda\coth(\lambda)-1}{1+\frac{1}{\sinh(\lambda)^2}}\sim\abs*{\lambda}$ for $\lambda\to\infty$, the second condition is also satisfied.
			
			\item[4.] {Just observe that} $\mathscr{G}$ is bounded in $L^\infty(\R)$. 
		\end{enumerate}
		{By} Theorem~\ref{stationaryphase}, we then get
		\begin{equation*}
			\begin{split}
				\int_\R e^{i R\psi_\omega(\lambda)} a_{k_1,k_2}(\lambda+i y_\omega)\,\dd  \lambda & = \frac{(2\pi) (4\pi)^n}{|x|}\Psi(\omega) a_{k_1,k_2}(iy_\omega)+ O\left(\frac{1}{\abs{x}^{3}}\right)
			\end{split}
		\end{equation*}
		for $R\to +\infty$, uniformly as $\omega$ runs through $[0,C]$. 
	 \end{proof}

	From now on, we shall consider the case $\omega\to +\infty$. The method of stationary phase cannot be applied directly in this case, since $y_\omega\to \pi$, and $i\pi$ is a pole of the phase (as well as of the amplitude). Although it seems possible to adapt the techniques developed by Li~\cite{Li3} to this situation, our proof follows the idea presented by Hueber and M\"uller~\cite[Theorem 1.3 (i)]{HueberMuller}. {We shall take} advantage of this singularity to get the correct behaviour of $h_{k_1,k_2}$, by means of the residues obtained by Lemma~\ref{stime:lem:8}.

	\subsection{\textbf{II.} Estimates for $\delta \rightarrow 0^+$ and $\kappa \rightarrow +\infty$.}
	We state below the main result of this section.
	
	\begin{theorem}\label{stimeII}
		For $\delta\to 0^+$ and $\kappa\to +\infty$
		\[
		p_{1,k_1,k_2}(x,t)= \frac{(-1)^{k_2} \pi^{k_1+k_2}}{4^{n}(\pi\delta)^{n+k_1-1}\sqrt{2\pi\kappa}} e^{-\frac{1}{4}d(x,t)^2} \left[1+O\left(\frac{1}{\kappa}+\delta \right)  \right].
		\]
		
	\end{theorem}
	The proof of Theorem~\ref{stimeII} will be prepared by several lemmata. The first step will be to invoke Lemma~\ref{stime:lem:8}, of which we keep the notation, to move the contour of integration beyond the singularity at $\pi i$; since at $2\pi i$ there is another one, it seems convenient {to stop at $\frac{3\pi i}{2}$}. We first notice that the integral {on $\R+\frac{3\pi i}{2}$}  may be neglected in some circumstances, as the following lemma shows. It is essentially~\cite[Lemma 1.4]{HueberMuller}, so we omit the proof.
	
	\begin{lemma}\label{stime:lem:3}
		There exists a constant $C'>0$ such that
		\[
		\abs*{\int_{\R}e^{i R\phi_\omega\left(\lambda+\frac{3\pi i}{2}\right)} a_{k_1,k_2}\left(\lambda+\dfrac{3\pi i}{2}\right)\,\dd  \lambda}\leq C'e^{-\frac{3\pi \abs{t}}{2}}.
		\]
	\end{lemma}
	Hence, {matters are reduced to} the computation of the residue. First of all, define
	\[
	r(\lambda) = \begin{cases} 1+\frac{1}{\lambda}-\pi(1+\lambda)\cot (\pi \lambda),	& \text{if $\lambda\not\in \Z$,}\\
	0, & \text{if $\lambda= 0$,}\end{cases}\\
	\]
	and observe that $r$ is holomorphic on its domain. It will be useful to define also 
	
	\[
	\tilde \phi_{k_1,k_2}(R,\xi )\coloneqq \begin{cases}
	e^{R \, r(-\xi )}\frac{(\pi \xi )^{n+k_1}\cos(\pi \xi )^{k_1}(1-\xi )^{n+k_1+k_2}}{\sin(\pi \xi )^{n+k_1}},  &\text{if $\xi \not \in  \Z$,}\\
	1,	 &\text{if  $\xi =0$,}
	\end{cases}
	\]
	and 
	\begin{equation}\label{phideltak1k2}
		\phi_{\delta, k_1,k_2}(s)\coloneqq e^{-i(n+k_1-1)s} \tilde \phi_{k_1,k_2}(0, \delta e^{i s})
	\end{equation}
	whenever $\delta e^{i s}\not \in\Z^*$. The following lemma may be proved again on the lines of~\cite[Lemma 1.4]{HueberMuller}.

	\begin{lemma}\label{stime:lem:11} For every $\delta<1$
		\begin{equation}\label{Hk1k2}
			2\pi i\,\res{e^{i R \phi_\omega} a_{k_1,k_2}}{\pi i}=\frac{(-1)^{k_2}\pi^{k_2+1} }{\delta^{n+k_1-1}} e^{-R-\pi\abs{t} } \int_{-\pi}^{\pi} e^{\kappa\cos(s)+R r(-\delta e^{i s}) }\phi_{\delta,k_1,k_2}(s)\,\dd  s.
		\end{equation}
	\end{lemma}
	Therefore, it remains only to estimate the integral in~\eqref{Hk1k2}, namely
	\begin{equation}\label{stime:eq:9}
		\begin{split}
			H_{k_1,k_2}(R,t)\coloneqq &\int_{-\pi}^{\pi} e^{\kappa\cos(s)+R r(-\delta e^{is}) }\phi_{\delta,k_1,k_2}(s)\,\dd  s =\int_{-\pi}^\pi e^{\kappa q_\delta( -i s)} \phi_{\delta,k_1,k_2}(s)\,\dd  s,
		\end{split}
	\end{equation}
	where 
	\begin{equation}\label{qdeltazeta}
		q(\delta,\zeta) = q_\delta(\zeta)\coloneqq \cosh(\zeta)+\frac{\delta}{2}r(-\delta e^{-\zeta}). 
	\end{equation} 
	Notice that we may apply Theorem~\ref{stationaryphase} only when $\kappa\to+\infty$, and this is why we confined ourselves to the case where $\delta\to 0^+$ (and we shall assume $0<\delta<1$)  and $\kappa\to+\infty$.
	
	Again for technical convenience, we shall reverse the dependence relation between $(\delta,\kappa)$ and $(R,\abs{t})$, thus assuming that $\delta$ and $\kappa$ are ``independent variables''. Indeed, $\delta$ and $\kappa$ completely describe our problem, since
	\[
	\abs{t}=\frac{\kappa}{2\pi\delta},\qquad R=\frac{\kappa \delta}{2}, 
	\]
	and $\abs{t}+R\to+\infty$ if $\delta\to 0^+$ and $\kappa\to+\infty$. We shall sometimes let $\delta$ assume complex values. The following lemma is essentially{~\cite[Lemma 1.2]{HueberMuller}}. We present a slightly shorter proof.
	\begin{lemma}	\label{stime:lem:6} 
		$q$ is holomorphic on the set $\{(\delta,\zeta)\in\C\times \C|\, \delta e^{-\zeta}\not\in\Z^*\}$. Moreover there exist two constants $\delta_1\in (0,1)$ and $\eta_1>0$ such that for all $\delta\in \mathrm{B}_\C(0,\delta_1)$ there is a unique $ \sigma_\delta \in \mathrm{B}_\C(0,\eta_1)$ such that $q'_\delta( \sigma_\delta )=0$. Then the mapping $\mathrm{B}_\C(0,\delta_1)\ni\delta\mapsto \sigma_\delta$ is holomorphic and real on $(-\delta_1,\delta_1)$. Finally, $\sigma_\delta = O(\delta^2)$ and $q_\delta(\sigma_\delta)=1+O(\delta^2)$ for $\delta\to 0$.
	\end{lemma}
	\begin{proof}
		$q$ is holomorphic since $r$ is. Furthermore, $\partial_2 q(0,0)=0$ and $\partial_2^2 q(0,0)=1$. Therefore, the implicit function theorem (cf.~\cite[Proposition 6.1 of IV.5.6]{Cartan}) implies the existence of some $\delta_1$ and $\eta_1$ as in the statement, the holomorphy of the mapping $\delta \mapsto \sigma_\delta$, and that $\frac{d}{d \delta} \sigma_\delta\vert_{\delta=0}=0$. Notice also that $\sigma_0=0$, so that $\sigma_\delta = O(\delta^2)$ for $\delta\to 0$ by Taylor's formula.
		
		Since $q_\delta$ is real on real numbers, $q'_\delta(\overline{\sigma_\delta})=\overline {q'_\delta(\sigma_\delta)}=0$; thus $\sigma_\delta=\overline{\sigma_\delta}$ for the uniqueness of $\sigma_\delta$, and hence $\sigma_\delta\in\R$ for all $\delta\in (-\delta_1,\delta_1)$.
		
		The last assertion follows from Taylor's formula, since $q_0(\sigma_0)=1$ and $\frac{d}{d \delta} q_\delta(\sigma_\delta)\vert_{\delta=0}= \partial_1 q(0,0)+\partial_2 q(0,0)\frac{d}{d \delta} \sigma_\delta\vert_{\delta=0}=0$.
	 \end{proof}
	{The contour of integration can now be changed} in order to apply the method of stationary phase. For the remainder of this section, we keep $\delta_1$ and $\eta_1$ of Lemma~\ref{stime:lem:6} fixed.
	
	\begin{lemma}\label{stime:lem:16}
		Let $\tau \in C_c^\infty(\R)$ such that $\chi_{[-\frac{\pi}{2}, \frac{\pi}{2}]}\leq \tau\leq \chi_{[\pi,\pi]}$. Define, for all $\delta\in (-\delta_1,\delta_1)$, the path  $\gamma_\delta(s)\coloneqq s+i\sigma_\delta\, \tau(s)$, and
		\begin{gather*}
			F_\delta (s)\coloneqq -i q_\delta (-i\gamma_\delta(s))+i q_\delta(\sigma_\delta) \qquad \text{and}\qquad
			\psi_{\delta,k_1,k_2}\coloneqq (\phi_{\delta,k_1,k_2}\circ \gamma_\delta)\, \gamma_\delta'.
		\end{gather*}
		Then
		\[
		H_{k_1,k_2}(R,t)= e^{\kappa \,q_\delta(\sigma_\delta)} \int_{-\pi}^{\pi}  e^{i\kappa  F_\delta\left(s\right)}\psi_{\delta,k_1,k_2}(s)\,\dd  s.
		\] 
	\end{lemma}

	\begin{proof}[Proof of Theorem \ref{stimeII}]
		We shall apply Theorem~\ref{stationaryphase}  to the bounded subsets  $\mathscr{F}= \{ F_\delta\colon \delta\in (0,\delta_2)\}$ and $\mathscr{G}= \{\psi_{\delta,k_1,k_2}\colon \delta\in (0,\delta_2)\}$ of $\Ec((-\pi,\pi))$, depending on some $\delta_2$ to be fixed later. Hence we check that the four conditions of the statement are satisfied.
		\begin{enumerate}
			\item [1.] The mapping $F\colon (-\delta_1,\delta_1)\times \R \ni (\delta,s )\mapsto  F_\delta(s )$ is of class $C^\infty$, and $\partial_2^2 F(0,0)=i$; thus we may find  $\delta_2\in (0,\delta_1)$, $\eta_2\in \left(0,\frac{\pi}{2}\right)$ and $C''>0$  such that $\Im \,\partial_2^2 F(\delta,s )\geq 2 C''$ for all $\delta\in [-\delta_2,\delta_2]$ and for all $s \in [-2\eta_2,2\eta_2]$. From Taylor's formula then 
			
			\[\Im\, F(\delta,s )=\int_0^s  \partial_2^2\, \Im\,  F(\delta,\tau)(s -\tau)\,\dd  \tau\geq C'' s ^2\]
			for all $s \in [-2\eta_2,2\eta_2]$ and for all $\delta\in [-\delta_2,\delta_2]$. {Since} $\Im \, F(0,s )=1-\cos(s )$ for all $s \in [-\pi,\pi]$,  by reducing $\delta_2$ and $C''$ if necessary one may assume that $\Im  \,F(\delta,s )\geq C''\pi^2\geq C'' s ^2$ for all $s \in \R$ such that $2\eta_2\leq \abs*{s}\leq \pi$   and for all $\delta\in [-\delta_2,\delta_2]$.
			
			\item [2.] It is immediately seen that $ F_\delta(0)= F_\delta'(0)=0$ by definition.
			
			\item[3.] For every $\delta\in [-\delta_2,\delta_2]$ and $s\in [-2 \eta_2,2 \eta_2]$
			\[\abs{\partial_2 F(\delta,s)}\geq\abs{\partial_2\,\Im \, F(\delta, s)}=\abs*{\int_{0}^s \partial_2^2\, \Im \,F(\delta,\tau)\,\dd \tau}\geq 2 C''\abs{s}.\] 
			\item[4.] Just observe that $\mathscr{G}$ is bounded in $L^\infty((-\pi,\pi))$.
		\end{enumerate}
		{By} Theorem \ref{stationaryphase}, then,
		\[ \begin{split}
		\int_{-\pi}^\pi e^{i\kappa F_\delta(s)}\tau_2(s)\psi_{\delta,k_1,k_2}(s)\,\dd s&=\sqrt{\frac{2\pi i}{\kappa F_\delta''(0) }} \psi_{\delta,k_1,k_2}(0)+O\left(\frac{1}{\kappa^{3/2}} \right)\end{split}. \]
		It is then easily seen that $ F_\delta''(0)= i q_\delta''(\sigma_\delta)=i(1 + O(\delta))$ and $\psi_{\delta,k_1,k_2}(0)=\phi_{\delta, k_1,k_2}(i\sigma_\delta)=1+O(\delta)$ for $\delta \rightarrow 0^+$.
		
		Now, by construction,
		\[
		-R-\pi \abs{t}+\kappa q_\delta(s)=i R\,\phi_\omega(\pi i (1-\delta e^{-s}))
		\]
		for $s$ in a neighbourhood of $\sigma_\delta$. Take $\delta_3\in (0,\delta_2]$ so that $ (1-\delta e^{-\sigma_\delta})\in (-1,1)$ for all $\delta\in [0,\delta_3]$, and fix $\delta\in (0,\delta_3)$ and {$t\neq 0$}. We shall prove that \[y_\omega= \pi (1- \delta e^{-\sigma_\delta}).\] Indeed, $y_\omega$ is the unique element of $(-\pi,\pi)$ such that $\phi'_\omega(i y_\omega)=0$; furthermore, $\pi(1- \delta e^{-\sigma_\delta})\in (-\pi,\pi)$ for the choice of $\delta_3$, and $-R\,\pi  \,\delta\, e^{-\sigma_\delta}\phi'_\omega(\pi i(1- \delta e^{-\sigma_\delta}))=\kappa\, q_\delta'(\sigma_\delta)=0$. Therefore, $y_\omega= \pi (1-\delta e^{-\sigma_\delta})$. Finally, equality holds by analyticity whenever both sides are defined. It then follows that
		\begin{equation}\label{esponente}
			-R-\pi \abs{t} +\kappa q_\delta(\sigma_\delta)=i R\phi_\omega(i y_\omega)=-\frac{1}{4}d(x,t)^2.
		\end{equation}
		Finally observe that, by definition of $\kappa$ and $\delta$, and by Lemma~\ref{stime:lem:6},
		\[
		-\frac{3\pi \abs{t}}{2}+R+\pi\abs{t}-\kappa q_\delta(\sigma_\delta) {+ \log \kappa } \leq -{\frac{\kappa}{2\pi \delta}}\left[ \frac{\pi}{2}- \pi\delta^2 + 2\pi \delta \left(1+O\left(\delta^2\right) \right) {- 2\pi \delta \frac{\log \kappa}{\kappa}} \right], 
		\]
		which tends to $-\infty$ as $\delta\to 0^+$ and $\kappa\to+\infty$. This means that
		\[e^{-\frac{3\pi \abs{t}}{2}}=o\left(\frac{e^{-R-\pi \abs{t}+\kappa q_\delta(\sigma_\delta)}}{\kappa}\right)\] for $\kappa\to +\infty$, uniformly as $\delta$ runs through $(0,\delta_2]$. Our assertion is then a consequence of~Lemmata~\ref{stime:lem:8} and~\ref{stime:lem:3}.
	 \end{proof}

	\subsection{\textbf{III} and \textbf{IV}. Estimates for $\delta\to 0^+$ and $\kappa$ bounded.}
Strictly speaking, cases {\bf III} and {\bf IV} have already been considered together by Hueber and M\"uller~\cite[Theorem 1.3 (ii)]{HueberMuller}. Despite this, we shall follow a different approach similar to that of Li~\cite{Li1}, which will allow us to get slightly better results.
	
	We first recall that, for all $\nu\in \Z$ and $\zeta\in\C$, the modified Bessel function $I_\nu$ of order $\nu$ is defined as
	\[
	I_\nu(\zeta)=\sum_{k\in\N} \frac{\zeta^{2k+\nu}}{2^{2 k+\nu}k!\,\Gamma(k+\nu+1)}.
	\]
	If $s>0$, then also
	\[
	I_\nu(s)=\frac{1}{2\pi}\int_{-\pi}^\pi e^{s\cos(\xi)-i \nu\xi}\,\dd  \xi,
	\]
	as one can verify from~\cite[7.3.1 (2)]{Erdelyi2} by applying the change of variables $\psi=\frac{\pi}{2}-\phi$ and by taking into account the relationship~\cite[7.2.2 (12)]{Erdelyi2} between $I_\nu=I_{-\nu}$ and $J_\nu$, and also the periodicity of the integrand. {Notice that for $s>0$ and $\nu \in \Z$, $I_\nu(s)$ is strictly positive unless $s=0$ and $\nu \neq 0$.} The main result of this section is the following.
	
	\begin{theorem}\label{stimeIIIeIV}
Fix $C>1$. If $\delta \to 0^+$ while {${1}/{C}\leq \kappa \leq C$}, then
\begin{equation}\label{stimeIII}
			{p_{1,k_1,k_2}(x,t) =\frac{(-1)^{k_2} \pi^{k_1+k_2}}{4^n (\pi\delta)^{n+k_1-1}} e^{-\frac{1}{4}d(x,t)^2 } e^{-\kappa } I_{{n+k_1-1}}(\kappa) \left[1 +O(\delta)\right].}
		\end{equation}
When $\kappa\to 0^+$ and $\abs{t}\to +\infty$
		\begin{equation}\label{stimeIV}
			\begin{split}
				p_{1,k_1,k_2}(x,t) &=\frac{(-1)^{k_2}\pi^{k_1+k_2}}{4^n(n+k_1-1)!} \abs*{t}^{n+k_1-1} e^{-\frac{1}{4}d(x,t)^2 }\left[1+O\left(\frac{1}{\abs{t}}+\kappa \right)  \right] . 
			\end{split}	
		\end{equation}	
	\end{theorem}
	\begin{lemma}\label{stime:lem:13}
		For every $N\in \N$
		\[
		H_{k_1,k_2}(R,t)=2\pi \sum_{\abs{\alpha}\leq N} I_{{n+k_1-1-\alpha_2}}(\kappa) \frac{\partial^\alpha\tilde \phi_{k_1,k_2}(0,0)\kappa^{\alpha_1}}{2^{\alpha_1}\alpha!}\delta^{\abs*{\alpha}}+O\left(\delta^{N+1}\right)
		\]
		for $\delta\to 0^+$,  uniformly as $\kappa$ runs through $[0, C]$. 
	\end{lemma}

	\begin{proof}
		By substituting \eqref{phideltak1k2} in \eqref{Hk1k2} and by Taylor's formula applied to $\tilde \phi_{k_1,k_2}$,
		\[
		\begin{split}
		H_{k_1,k_2}(R,t)&=\int_{-\pi}^\pi e^{\kappa \cos(s)}e^{-i(n+k_1-1)s}\tilde \phi_{k_1,k_2}(R,\delta e^{is})\,\dd s\\
		&=\sum_{\abs*{\alpha}\leq N} \frac{\partial^{\alpha}\tilde \phi_{k_1,k_2}(0,0)}{\alpha!}R^{\alpha_1}\delta^{\alpha_2} \int_{-\pi}^\pi e^{\kappa \cos(s)} e^{-i(n+k_1-1-\alpha_2) s} \,\dd s + \mathcal{R}_{N+1}(\delta,\kappa)
		\\&= 2\pi\sum_{\abs*{\alpha}\leq N} I_{{n+k_1-1-\alpha_2}}(\kappa) \frac{\partial^\alpha\tilde \phi_{k_1,k_2}(0,0)\kappa^{\alpha_1}}{2^{\alpha_1}\alpha!}\delta^{\abs*{\alpha}}+ \mathcal{R}_{N+1}(\delta,\kappa),
		\end{split}
		\]
		where the last equality holds since $R=\frac{\delta\kappa}{2}$. Moreover, $\mathcal{R}_{N+1}(\delta,\kappa)$ is easily seen to be $O\left(\delta^{N+1}\right)$ for $\delta\to 0^+$ uniformly as $\kappa$ runs through $[0,C]$. This completes the proof.
	 \end{proof}

	\begin{proof}[Proof of Theorem~\ref{stimeIIIeIV}]
		Lemmata~\ref{stime:lem:3} and~\ref{stime:lem:11} imply that
		\[
		p_{1,k_1,k_2}(x,t)\!=\!\frac{(-1)^{k_2}2 \pi^{k_2-n}}{4^{n+1}\delta^{n+k_1-1}} e^{-R-\pi\abs*{t}} H_{k_1,k_2}(R,t)+ O\left(e^{-\frac{3\pi \abs*{t}}{2}}\right).
		\]
		Moreover, recall that $\delta \abs*{t}= \frac{\kappa}{2 \pi}$ and $R=\frac{\kappa \delta}{2}$; therefore, for every $N\in\N$,  
		\begin{equation}\label{restoIII}
			e^{-\frac{3\pi \abs*{t}}{2}}= o\left(\delta^{N+2-n-k_1}e^{-R-\pi\abs*{t}}\right)
		\end{equation}
		as $\delta\to 0^+$, uniformly as $\kappa$ runs through $[1/C,C]$. By~\eqref{esponente} and Lemma~\ref{stime:lem:6}, the first assertion follows from Lemma~\ref{stime:lem:13} for $N=0$.
		
		As for \eqref{stimeIV}, observe first that $\kappa \to 0^+$ and $\abs{t} \to +\infty$ is equivalent to saying $\delta,\kappa \to 0^+$ and $\delta =o(\kappa)$. Then Lemma~\ref{stime:lem:13} with $N=n+k_1-1$  and an easy development of the Bessel function in a neighbourhood of $0$ imply that
		\begin{multline*}
			p_{1,k_1,k_2}(x,t)=  \frac{\pi^{k_1+k_2}(-1)^{k_2}}{4^n (\pi\delta)^{n+k_1-1}} e^{-\pi \abs{t}-R} \Bigg[  \kappa^{n+k_1-1} \frac{I_{n+k_1-1}^{(n+k_1-1)}(0)}{(n+k_1-1)!}+O(\kappa^{n+k_1})\\
			+\sum_{1\leq \abs{\alpha}\leq n+k_1-1} O\left(I_{n+k_1-1-\alpha_2}(\kappa) \kappa^{\alpha_1}\delta^{\abs{\alpha}}\right)+O(\delta^{n+k_1})\Bigg] + O\left(e^{-\frac{3\pi \abs{t}}{2}}\right).
		\end{multline*}
		Since $\delta=o(\kappa)$, one has $\delta^{\alpha_2+\alpha_1-1}=O(\kappa^{\alpha_2+\alpha_1-1})$  for every $\alpha\neq 0$. Therefore,
		\[
		\begin{split}
		\sum_{1\leq \abs{\alpha}\leq n+k_1-1} O\left(I_{n+k_1-1-\alpha_2}(\kappa) \kappa^{\alpha_1}\delta^{\abs{\alpha}}\right)		&=\sum_{1\leq \abs{\alpha}\leq n+k_1-1} O\left(\kappa^{n+k_1-2+2\alpha_1}\delta\right)\\
		&=O\left(\kappa^{n+k_1-2}\delta\right).
		\end{split}
		\]
		Since $\frac{\kappa}{2\pi \delta}= \abs{t}$ and $I_{n+k_1-1}^{(n+k_1-1)}(0)= \frac{1}{2^{n+k_1-1}}$, we get
		\[
		p_{1,k_1,k_2}(x,t)=  \frac{\pi^{k_1+k_2} (-1)^{k_2}}{4^n(n+k_1-1)!} e^{-\pi \abs{t}-R} \abs{t}^{n+k_1-1}\left[  1+ O\left(\frac{1}{\abs{t}}+\kappa+\delta\right)\right] + O\left(e^{-\frac{3\pi \abs{t}}{2}}\right).
		\]
		{Finally,  $\delta=o\left(\frac{1}{\abs{t}}\right)$ since $\delta \abs{t}= \frac{\kappa}{2\pi}$; moreover}
		\[
		e^{-\frac{3\pi \abs{t}}{2}}=o\left( e^{-\pi \abs{t}-R} \abs{t}^{n+k_1-2}\right) 
		\]
		since $R\to 0^+$ and $\abs{t}\to +\infty$. The assertion follows.
	 \end{proof}
	
	The estimates in cases {\bf II}, {\bf III}, and {\bf IV} can be put together. This is done in the following corollary, which will turn out to be fundamental later on. Define first, for $\zeta\in \C$ and $\nu\in \Z$,
	\[
		\tilde I_\nu(\zeta)\coloneqq \sum_{k\geq 0}  \frac{\zeta^{2 k}}{2^{2 k+\nu}k! \Gamma(k+\nu+1)}.
	\]
	From now on we shall use the following abbreviation. We keep the notation of Lemma~\ref{stime:lem:6}.
	
	\begin{definition}
		For $\delta\in B_\C(0,\delta_1)$, define $\rho(\delta)\coloneqq q_\delta(\sigma_\delta)$.
	\end{definition}
	
	By Lemma~\ref{stime:lem:6}, $\rho$ is a holomorphic function such that $\rho(0)=1$ and $\rho'(0)=0$, so that $\rho(\delta)=1+O(\delta^2)$ as $\delta\to 0$.
	\begin{corollary}\label{cor:1}
		When $(x,t)\to \infty$ and $\delta\to 0^+$
		\begin{equation*}
			p_{1,k_1,k_2}(x,t)= \frac{(-1)^{k_2}\pi^{k_1+k_2} }{2^{n -k_1+1} }\abs*{t}^{n+k_1-1}  e^{-\frac{1}{4}d(x,t)^2}e^{-\kappa \rho(\delta)}\tilde I_{n+k_1-1}\left( \kappa \rho(\delta)\right) \left[1+ g(\abs{x},\abs{t})  \right],
		\end{equation*}
		where
		\begin{equation}\label{remainders}
			g(\abs{x},\abs{t})= \begin{cases}
				O\left(\delta+\frac{1}{\kappa}\right) &\text{if $\delta\to 0^+$ and $\kappa\to +\infty$,}\\
				O(\delta) &\text{if $\delta\to 0^+$ and {$\kappa\in [1/C,C]$},}\\
				O\left(\frac{1}{\abs{t}}+\kappa\right) &\text{if $\delta\to 0^+$ and $\kappa\to 0^+$}
			\end{cases}
		\end{equation}
		for every $C>1$.
	\end{corollary}
	
	\begin{proof}
		{\bf1.} Assume first that $\kappa\to +\infty $. Since $ I_\nu(s)= \frac{e^s}{\sqrt{2\pi s}}\left[ 1+O\left( \frac{1}{s}\right) \right]$ for $s\to +\infty$, $\nu\in\Z$ (cf.~\cite[7.13.1 (5)]{Erdelyi2}),
		\begin{equation}	\label{modibess}
			{\tilde I_\nu(s)=\frac{e^s}{s^\nu\sqrt{2\pi s}}\left[ 1+O\left( \frac{1}{s}\right) \right] \quad \mbox{for $s\to +\infty$.}}
		\end{equation}
		Therefore, Theorem~\ref{stimeII} implies that
{		\[
		\begin{split}
		p_{1,k_1,k_2}(x,t)&= \frac{(-1)^{k_2}\pi^{k_1+k_2} }{4^n (\pi\delta)^{n+k_1-1}\sqrt{2\pi \kappa}}e^{-\frac{1}{4}d(x,t)^2 } \left[1+O\left(\frac{1}{\kappa} +\delta \right)\right] \\
		&=\frac{(-1)^{k_2}\pi^{k_1+k_2} \tilde I_{n+k_1-1}\left( \kappa \rho(\delta)\right) }{2^{n -k_1+1} }\abs*{t}^{n+k_1-1}  e^{-\frac{1}{4}d(x,t)^2 } e^{-\kappa \rho(\delta)} \\
		&\qquad\qquad\qquad\times\left[1+O\left(\frac{1}{\kappa\rho(\delta)}  \right)\right] \left[1+O\left(\frac{1}{\kappa} +\delta \right)\right]\\
		&=\frac{(-1)^{k_2}\pi^{k_1+k_2} \tilde I_{n+k_1-1}\left( \kappa \rho(\delta)\right) }{2^{n -k_1+1} }\abs*{t}^{n+k_1-1} e^{-\frac{1}{4}d(x,t)^2 } e^{-\kappa \rho(\delta)} \left[1+O\left(\frac{1}{\kappa} +\delta \right)\right],
		\end{split}
		\]}
		since $\rho(\delta)=1+O(\delta^2)$ and $\frac{2 \abs{t}}{\kappa}=\frac{1}{\pi\delta}$.
		
		{\bf2.} Assume now that ${\kappa\in [1/C,C]}$ for some $C>1$.  Then, by Theorem~\ref{stimeIIIeIV},
\[
		\begin{split}
		p_{1,k_1,k_2}(x,t)&=  \frac{(-1)^{k_2}\pi^{k_1+k_2}}{4^n (\pi\delta)^{n+k_1-1}} e^{-\frac{1}{4}d(x,t)^2 } e^{-\kappa}  I_{n+k_1-1}(\kappa) \left[1+O\left(\delta \right)\right] \\
		&=\frac{(-1)^{k_2}\pi^{k_1+k_2}}{2^{n -k_1+1}}\abs{t}^{n+k_1-1} e^{-\frac{1}{4}d(x,t)^2 } e^{-\kappa \rho(\delta) } \tilde I_{n+k_1-1}(\kappa \rho(\delta)) \left[1+O\left(\delta^2 \right)\right] \left[1+O\left(\delta  \right)\right]\\
		&=\frac{(-1)^{k_2}\pi^{k_1+k_2} }{2^{n -k_1+1}}\abs{t}^{n+k_1-1} e^{-\frac{1}{4}d(x,t)^2 } e^{-\kappa \rho(\delta)}  \tilde I_{n+k_1-1}(\kappa\rho(\delta)) \left[1+O\left(\delta\right)\right],
		\end{split}
		\]
		where the second equality holds since $I_{n+k_1-1}(\kappa \rho(\delta))-I_{n+k_1-1}(\kappa)=O(\kappa (\rho(\delta)-1))=O(\delta^2)$ uniformly as $\kappa$ runs through $ [1/C,C]$ by Taylor's formula.
		
		{\bf3.} Finally, if $\kappa\to 0^+$ then
		\[\tilde I_{n+k_1-1}(\kappa)=\tilde I_{n+k_1-1}(0)+O(\kappa)= \frac{1}{2^{n+k_1-1}(n+k_1-1)!} + O(\kappa)\] by the definition of $\tilde I_{n+k_1-1}$. Combining this estimate with  Theorem~\ref{stimeIIIeIV} yields the assertion.
	 \end{proof}

	\section{H-type Groups}\label{sec:mgeneral}
	In this section we deal with the general case $m\geq1$. In particular, we prove a refined version of Theorem \ref{HeisHstimeI}, and extend Theorems~\ref{stimeII} and \ref{stimeIIIeIV}: this is done through Theorems \ref{HstimeI}, \ref{HstimeII} and \ref{HstimeIIIeIV} respectively. {Theorem~\ref{HstimeI} treats the case \textbf{I} and is still inspired by~\cite[Theorem 2 of § 3]{Gaveau}.  The asymptotic estimates in the other three cases are first obtained in the case $m$ odd, ``reducing'' to the case $m=1$; the case $m$ even is then achieved through a descent method.}
	
	The first step in order to apply the method of stationary phase is to extend the integrand to a meromorphic function on $\C^m$. If $m>1$, such extension is no longer automatic as when $m=1$. A natural way consists in taking advantage of the parity of the functions that appear, as in~\cite{Eldredge}. Indeed, any continuous  branch of $\lambda \mapsto\sqrt{\lambda^2}$ is a holomorphic function which coincides with $\lambda\mapsto\pm\abs{\lambda}$ on $\R^m$; therefore, whenever $g$ is an even holomorphic function defined on a symmetric open subset of $\C$, the function $\lambda\mapsto g(\sqrt{\lambda^2})$ is well-defined, holomorphic, and coincides with $\lambda \mapsto g(\abs{\lambda})$ on {$\R^m$}. Hence, we are led to the following definition, which is the analogue of Definition \ref{hk1k2}. We shall use the same notation as before, without stressing the (new) dependence on $m$.

	\begin{definition} 
		Define
		\begin{equation*}\label{Hhk1k2}
			\begin{split}
				h_{k_1,k_2}(R,t)= \int_{\mathbb{R}^m} e^{i R\phi_\omega(\lambda)}a_{k_1,k_2}(\lambda)\,\dd \lambda \end{split}
		\end{equation*}
		where
		\begin{equation}\label{Hak1k2}
			\begin{split}
				a_{k_1,k_2}(\lambda)&=\begin{cases}
					(-1)^{k_1}i^{k_2}\frac{\sqrt{\lambda^2}^{n+k_1  }\cosh(\sqrt{\lambda^2})^{k_1}}{\sinh(\sqrt{\lambda^2})^{n+k_1}}(\lambda, u_1)^{k_2} & \text{if $\sqrt{\lambda^2}\not\in i \pi \Z^*$,}\\
					(-1)^{k_1}i^{k_2}\delta_{k_2,0}	&\text{if $\lambda  =0$,}
				\end{cases}  	\\
				\phi_\omega  (\lambda)&= \begin{cases}
					\omega\, (\lambda, u_1) +i\sqrt{\lambda^2}  \coth(\sqrt{\lambda^2} ) & \text{if $\sqrt{\lambda^2}\not\in i \pi \Z^*$,}\\
					i	&\text{if $\lambda  =0$}.
				\end{cases}
			\end{split} 
		\end{equation}
Define also
\begin{equation}\label{ak1k2omega}
a_{k_1,k_2,\omega}(\lambda)  \coloneqq a_{k_1,k_2}(\lambda+iy_\omega u_1).
\end{equation}
	\end{definition}
	
	Observe again that
	\[
	p_{1,k_1,k_2}(x,t)= \frac{1}{(4\pi)^n(2\pi)^{m}} h_{k_1,k_2}\left(R,t\right)
	\]
	for all $(x,t)\in \R^{2n}\times \R^m$, and that $y_\omega =\theta^{-1}(\omega) \in [0,\pi)$, since $\omega\geq 0$.
	
	\subsection{\textbf{I.} Estimates for $(x,t)\to \infty$ while $ 4\abs{t}/\abs{x}^2 \leq C$.}\label{SecStimeIm>1}

	 The main result of this section is Theorem~\ref{HstimeI} below. As already said, the main ingredient of its proof is the method of stationary phase (cf.~Proposition~\ref{propHorm}), which is already employed in~\cite[Theorem 2 of § 3]{Gaveau} to treat the case $n=m=1$ and $k_1=k_2=0$.
	 
The novelty of considering all the derivatives of the heat kernel $p_1$ (in other words, all the cases $k_1\geq 0$ and $k_2\geq 0$) {introduces} additional complexity to the developments, since {the choice $k=0$ in~\eqref{eqfasestaz}} may not give the  sharp asymptotic behaviour of $p_{1,k_1,k_2}$ at infinity, while $\omega$ remains bounded. In particular, this happens in the cases $\omega \to 0$ and $k_2>0$, or $\omega \to \frac{\pi}{2}$ and $k_1>0$. If {$\omega$ remains bounded and away from $0$ and $\frac{\pi}{2}$}, the first term is instead enough.
	\begin{theorem}\label{HstimeI}
Fix $\varepsilon, C>0$. If $(x,t)\to \infty$ while $0\leq \omega \leq C$, then
\[
p_{1,k_1,k_2}(x,t)= \frac{1}{|x|^m}e^{-\frac{1}{4}d(x,t)^2}\Psi(\omega) \Upsilon(x,t)
\]
where
\begin{equation}\label{PsiI}
\Psi(\omega)=\begin{cases}
	\frac{1}{4^n \pi^{n+m}}\sqrt{\frac{(2 \pi)^m y_\omega^{m-1}\sin(y_\omega)^3}{2\omega^{m-1}(\sin(y_\omega)-y_\omega\cos(y_\omega))}}, 	& \text{if $\omega\neq 0$,}\\
		\frac{(3\pi)^{m/2}}{4^n \pi^{n+m}}, & \text{if $\omega= 0$,}\end{cases}
\end{equation}
and
\begin{itemize}
\item[1.] if $\varepsilon\leq \omega \leq \frac{\pi}{2}-\varepsilon$ or $\frac{\pi}{2}+\varepsilon \leq \omega \leq C$,
\begin{equation}\label{HstimeI_1}
\Upsilon(x,t)= (-1)^{k_1+k_2}  \frac{ y_\omega^{n+k_1+k_2}\cos(y_\omega)^{k_1}}{\sin (y_\omega)^{n+k_1}} +O\left(\frac{1}{|x|^2}\right);
\end{equation}
\item[2.] if $\omega \to 0$ and $k_2$ is even,
\begin{equation}\label{HstimeI_2}
\Upsilon(x,t)= \sum_{j=0}^{k_2/2}c_{k_1,k_2,j} \frac{\omega^{k_2-2j}}{|x|^{2j}} + O\left(\sum_{j=0}^{k_2/2} \frac{\omega^{k_2-2j+1}}{|x|^{2j}}+\frac{1}{|x|^{k_2+2}}\right);
\end{equation}
\item[3.] if $\omega \to 0$, $k_2$ is odd and $|t|\to \infty$,
\begin{equation}\label{HstimeI_3}
\Upsilon(x,t)= \sum_{j=0}^{(k_2-1)/2} c_{k_1,k_2,j}\frac{ \omega^{k_2-2j}}{|x|^{2j}} + O\left(\sum_{j=0}^{(k_2+1)/2} \frac{ \omega^{k_2-2j+1}}{|x|^{2j}}\right);
\end{equation}
\item[4.]  if $\omega \to 0$, $k_2$ is odd and $0\leq |t|\leq C$
\begin{equation}\label{HstimeI_4}
\Upsilon(x,t)=c_{k_1,k_2+1, (k_2+1)/2}\frac{|t|}{|x|^{k_2+1}} +O\left(\frac{|t|}{|x|^{k_{2}+3}}\right);
\end{equation}
\item[5.] if $\omega \to \frac{\pi}{2}$ and $k_1$ is even,
\begin{equation}\label{HstimeI_5}
\Upsilon(x,t) = \sum_{j=0}^{k_1/2} b_{k_1,k_2,j} \frac{\left(\omega -\frac{\pi}{2}\right)^{k_1-2j}}{|x|^{2j}} + O\left( \sum_{j=0}^{k_1/2}\frac{\left(\omega -\frac{\pi}{2}\right)^{k_1-2j+1}}{|x|^{2j}} + \frac{1}{|x|^{k_1+2}}\right);
\end{equation}

\item[6.] if $\omega \to \frac{\pi}{2}$ and $k_1$ is odd,
\begin{multline}\label{HstimeI_6}
\Upsilon(x,t) = \sum_{j=0}^{(k_1-1)/2} b_{k_1,k_2,j} \frac{\left(\omega -\frac{\pi}{2}\right)^{k_1-2j}}{|x|^{2j}} + \frac{b_{k_1,k_2,(k_1+1)/2}}{|x|^{k_1+1}}\\+ O\left( \sum_{j=0}^{(k_1-1)/2}\frac{\left(\omega -\frac{\pi}{2}\right)^{k_1-2j+1}}{|x|^{2j}}+\frac{\omega-\frac{\pi}{2}}{\abs{x}^{k_1+1}}+\frac{1}{\abs{x}^{k_1+3}} \right).
\end{multline}
\end{itemize}
The coefficients $c_{k_1,k_2,j}$ and $b_{k_1,k_2,j}$ are explicitly given by~\eqref{ck1k2},~\eqref{bk1k2j} and~\eqref{bk1k2odd}.
		\end{theorem}
	The remainder of this section is devoted to the proof of Theorem~\ref{HstimeI}. Since it is quite involved, we split this section into two parts: in the first one we apply the method of stationary phase, while in the second one we find the asymptotics of the development given by Theorem~\ref{stationaryphase} which are required to get the sharp developments~\eqref{HstimeI_2}--\eqref{HstimeI_6}. These proofs go through several lemmata.

\begin{remark}
Notice that any {pair of terms} in the sums {appearing} in the developments~\eqref{HstimeI_2}, \eqref{HstimeI_3}, \eqref{HstimeI_5}, and~\eqref{HstimeI_6} are not comparable {with each other} under the stated asymptotic condition. Therefore, these developments cannot be simplified. Observe moreover that for $k_1$ and $k_2$ fixed the coefficients $b_{k_1,k_2,j}$ (resp.\ $c_{k_1,k_2,j}$) have the same sign; thus, no cancellation can occur, and our developments are indeed \emph{sharp}. {A more detailed description will be given in Section~\ref{sec:further}}.

Finally, notice that {it is possible to} obtain even more precise {expansions} if one does not develop the terms $L_{j,\psi_\omega}a_{k_1,k_2,\omega}$ which appear in Proposition~\ref{propHorm} {below}. In particular, in the cases when $\omega\to 0^+$ and $k_2=0$, or $\omega\to \frac{\pi}{2}$ and $k_1=0$, the explicit computation of $L_{0,\psi_\omega}a_{k_1,k_2,\omega} = a_{k_1,k_2}(iy_{\omega}u_1)$ leads to better remainders than those in~\eqref{HstimeI_2} and~\eqref{HstimeI_5} respectively.
\end{remark}

\subsubsection{Application of the Method of Stationary Phase}
As already said, Proposition~\ref{propHorm} below is essentially an easy generalization of Theorem \ref{HeisHstimeI}.
\begin{proposition}\label{propHorm}
Fix $C>0$ and let $k \in \N$. Then, if $(x,t)\to \infty$ while $0\leq \omega \leq C$,
\begin{equation}\label{eq:HstimeI}
			p_{1,k_1,k_2}(x,t) = \frac{1}{\abs{x}^m}e^{-\frac{1}{4}d(x,t)^2}\Psi(\omega)\left[ \sum_{j=0}^{k} \frac{4^j L_{j,\psi_\omega} a_{k_1,k_2,\omega}}{|x|^{2j}} + O\left(\frac{1}{\abs{x}^{2k +2}}\right)\right]
		\end{equation}		
{where $\Psi$ is defined by~\eqref{PsiI}.}
\end{proposition}	
	In the same way as in Section~\ref{sec_HeisI}, we begin by finding some stationary points of the phase of $h_{k_1,k_2}$, namely $\phi_\omega$.

	\begin{lemma}\emph{\cite[Formula (5.7)]{Eldredge}}\label{Hstime:lem:7}
		For all $\lambda$ such that $\sqrt{\lambda^2}\not\in i \pi \Z^*$,\[\phi_\omega'(\lambda)= \omega u_1 + \lambda\frac{\tilde \theta(i \sqrt{\lambda^2})}{\sqrt{\lambda^2}}\]
		where $\tilde\theta$ is the analytic continuation of $\theta$ to $\Dom(\phi_{\omega})$. In particular, $i y_\omega u_1$ is a stationary point of $\phi_\omega$.
	\end{lemma}
	We then change the contour of integration in the integral defining $h_{k_1,k_2}$ in order to meet a stationary point of $\phi_\omega$. This is done in the following lemma, which is the analogue of Lemma~\ref{stime:lem:8}.

	\begin{lemma}\label{Hstime:lem:8}
		For every $y\in [0,\pi)$
		\[
		\begin{split}
		h_{k_1,k_2}(R,t)= \int_{\R^m} e^{i R\phi_\omega(\lambda+ i y u_1)} a_{k_1,k_2}(\lambda+i y u_1)\,\dd  \lambda.
		\end{split}
		\]
	\end{lemma}
	
	\begin{proof}
		The theorem is proved in a similar fashion to \cite[Lemma~5.4]{Eldredge}. It may be useful to observe that  for every $\lambda\in \C^m$ such that either $\Im\sqrt{\lambda^2}\notin \pi \Z$ or $\Re\sqrt{\lambda^2}\neq 0$, we have
		\[
		\abs{a_{k_1,k_2}(\lambda)}= \frac{\abs{\lambda}^{n+k_1} \left(\sinh\left(\Re\sqrt{\lambda^2}\right)^2+\cos\left( \Im\sqrt{\lambda^2}\right)^2\right)^{k_1/2}   }{\left(\sinh\left(\Re\sqrt{\lambda^2}\right)^2+\sin\left( \Im\sqrt{\lambda^2}\right)^2\right)^{(n+k_1)/2}}\abs*{(\lambda,u_1)}^{k_2},
		\]
by Lemma~\ref{stime:lem:5}, since $\abs{\sqrt{\lambda^2}}=\abs{\lambda}$. Moreover,  $a_{k_1,k_2}$ is bounded on the set $\{\lambda+i y u_1 \colon \lambda\in \R^m, y\in [0,C']\}$ for every $C'\in (0,\pi)$.
	 \end{proof}

	\begin{proof}[Proof of Proposition~\ref{propHorm}] 
		Define
		\[
		\psi_\omega\:= \phi_\omega(\,\cdot\,+i y_\omega u_1)-\phi_\omega(i y_\omega u_1)
		\]
		and observe that, since $\sqrt{(i y_\omega u_1)^2}=\pm i y_\omega$ and  $\omega=\theta(y_\omega)$, $\phi_\omega(i y_\omega u_1) =i\frac{y_\omega^2}{\sin(y_\omega)^2}$. Therefore, by Lemma~\ref{Hstime:lem:8}
		\[
		h_{k_1,k_2}(R,t)=e^{-\frac{1}{4}d(x,t)^2} \int_\R e^{i R\psi_\omega(\lambda)}a_{k_1,k_2}(\lambda+i y_\omega u_1)\,\dd  \lambda.
		\]
		We shall apply Theorem~\ref{stationaryphase} to the bounded subsets $\mathscr{F}=\{\psi_\omega \colon \omega\in [0, C]\}$ and $\mathscr{G}= \{a_{k_1,k_2,\omega}\colon \omega\in [0,C]\}$ of $\Ec(\R^m)$.
		\begin{enumerate}
			\item[2.] Elementary computations show that
\begin{equation}\label{psider2}
			-i\psi''_\omega(0)= \theta'(y_\omega)u_1 \otimes u_1+ \frac{\omega}{y_\omega} \sum_{j=2}^m u_j\otimes u_j,
\end{equation}
			so that $ \det(-i\psi''_\omega(0)) = \theta'(y_\omega)\left( \frac{\theta(y_w)}{y_w}\right) ^{m-1}>0$. The conditions $	\psi_\omega(0)=\psi'_\omega(0)=0$ hold by construction.
			
			\item[3.] Consider the mapping  $\psi\colon \R^m\times (-\pi,\pi)\ni(\lambda,y)\mapsto \psi_{\theta(y)}(\lambda)$. Then, by the preceding arguments, there is $c>0$ such that $\partial_1\psi(0,y)=0$ and $-i \partial_1^2 \psi(0,y)\geq c (\,\cdot\,,\,\cdot\,)$ for all $y\in [0,\pi)$; moreover, $\psi$ is analytic by Lemma~\ref{stime:lem:10}. Therefore, by Taylor's formula we may find two constants $\eta >0$ and $C'>0$ such that $ \abs{\partial_1\psi(\lambda,y)}\geq C' \abs*{\lambda}$ for all $\lambda\in B_{\R^m}(0,2\eta)$ and for all $y\in [0,\theta^{-1}(C)]$.
			
			\item[1.] Combining~\cite[Lemmata 5.3 and 5.7]{Eldredge}, we infer that there is a constant $C''>0$ such that
			\[\begin{split}
			\Im\,\psi(\lambda,y)&= y\theta(y) + \Re\left[\sqrt{(\lambda +i y u_1)^2}\coth\sqrt{(\lambda +i y u_1)^2}\right] - \frac{y^2}{\sin^2 y} 
			\geq C''\abs*{\lambda}
			\end{split}\]
			whenever $\abs{\lambda}\geq \eta$ and $0\leq y\leq \theta^{-1}(C)$.
			\item[4.]  Just observe that $\mathscr{G}$ is bounded in $L^\infty(\R^m)$.
		\end{enumerate}
By Theorem~\ref{stationaryphase}, then,
		
		\begin{equation*}
			\begin{split}
				\int_{\R^m}  e^{i R\psi_\omega(\lambda)}a_{k_1,k_2}(\lambda+i y_\omega u_1)\,\dd  \lambda 
				= \frac{(2\pi)^m (4\pi)^n}{|x|^m}\Psi(\omega)  \sum_{j=0}^k \frac{4^j L_{j,\psi_\omega}a_{k_1,k_2,\omega}}{\abs{x}^{2j}}+ O\left(\frac{1}{\abs{x}^{m+2k+2}}\right)
			\end{split}
		\end{equation*}	
		for $R\to +\infty$,	uniformly as $\omega$ runs through $[0,C]$. 
	 \end{proof}

	\subsubsection{Further Developments and Completion of the Proof of Theorem~\ref{HstimeI}}\label{sec:further}
We begin by recalling that, for every $j\in \N$,
\begin{equation}\label{Lj}
	L_{j,\psi_\omega}a_{k_1,k_2,\omega}= i^{-j}\sum_{\mu=0}^{2j}\frac{ (\psi_\omega''(0)^{-1} \partial, \partial )^{\mu+j} [ (\psi_\omega-P_{2,0}\psi_\omega )^\mu a_{k_1,k_2,\omega} ](0)}{ 2^{\mu+j} \mu !(\mu+j) !}.
\end{equation}	
Thus, the point 1 of Theorem~\ref{HstimeI} follows immediately by taking $k=0$ in Proposition~\ref{propHorm}, since
\[L_{0,\psi_\omega}a_{k_1,k_2,\omega} = a_{k_1,k_2,\omega}(0) = a_{k_1,k_2}(iy_{\omega}u_1).\]
As for the other developments, observe that by~\eqref{psider2}
\begin{align} \label{psi2partialpartial}
		&( \psi_\omega''(0)^{-1}\partial, \partial )^{\mu+j} [ (\psi_\omega-P_{2,0}\psi_\omega)^\mu a_{k_1,k_2,\omega} ](0) \nonumber \\
		&\qquad=\sum_{\abs{\alpha}=\mu+j  } \frac{ (\mu+j)! }{\alpha !} \frac{1}{ (i \theta'(y_\omega))^{\alpha_1} } \left( \frac{y_\omega}{i\omega}   \right)^{\abs{\alpha}-\alpha_1} \partial^{2\alpha}  [ (\psi_\omega-P_{2,0}\psi_\omega)^\mu a_{k_1,k_2,\omega} ] (0).
\end{align}
	where
	\begin{multline}\label{omega0mu}
			\partial^{2 \alpha} [ (\psi_\omega-P_{2,0}\psi_\omega)^\mu a_{k_1,k_2,\omega} ] (0)\\= \sum_{\substack{\beta\leq 2 \alpha,\\ \abs{\beta}\geq 3 \mu}} \frac{(2 \alpha)!}{\beta !\, (2\alpha-\beta) !}\partial^{\beta}[(\psi_\omega-P_{2,0}\psi_\omega)^\mu](0)\, \partial^{2\alpha-\beta} a_{k_1,k_2}(i y_\omega u_1).
	\end{multline}
	The sum above is restricted to $|\beta|\geq 3\mu$ since $\psi_\omega(\lambda)- P_{2,0}\psi_\omega(\lambda)$ is infinitesimal of order at least $3$ for $\lambda\to 0$. {Observe moreover that, since $\abs{2\alpha-\beta}= 2\abs{\alpha}-\abs{\beta}\leq 2 j-\mu$, we have $\abs{2\alpha-\beta}\leq 2j$ and $\abs{2\alpha-\beta}=2j$ if and only if $\mu=0$ and $\beta=0$.} We first consider the case $\omega \to 0$.

	\begin{lemma}\label{lemmaHomega0}
		For every $j\in\N$ such that $2 j \leq k_2$, define
		\begin{equation}\label{ck1k2}
c_{k_1,k_2,j}\coloneqq (-1)^{k_1+k_2} \frac{3^{k_2-j}k_2 !}{2^{k_2-2j}(k_2-2 j)!  j!}		.
		\end{equation}
		Then
		\[
			4^j L_{j,\psi_\omega} a_{k_1,k_2,\omega}= c_{k_1,k_2,j}\omega^{k_2-2j} + O\left(\omega^{k_2-2 j+1}  \right)
		\]
		for $\omega\to 0$.
	\end{lemma}
	
	\begin{proof}
		Recall that $a_{k_1,k_2}$ is an analytic function on its domain, and observe that\footnote{Here and in the following, $\lambda_1$ stands for $(\lambda, u_1)$.}
\begin{equation*}
		a_{k_1,k_2}(\lambda)=(-1)^{k_1}i^{k_2}\lambda_1^{k_2}+O\left(\abs{\lambda}^{k_2+2}\right)
\end{equation*}
		for $\lambda\to 0$. Therefore, for every $h=0,\dots, k_2$ we have
		\begin{equation}\label{asympak1k2in0}
			a_{k_1,k_2}^{(h)}(\lambda)= (-1)^{k_1}i^{k_2}\frac{k_2 !}{(k_2-h)!} \lambda_1^{k_2-h} u_1^{\otimes h}  + O\left(\abs{\lambda}^{k_2-h+2}  \right)  
			\end{equation}
	as $\lambda\to 0$. 

We now consider~\eqref{omega0mu}. If $\abs{2\alpha-\beta} < 2 j$, then by~\eqref{asympak1k2in0}
\[
\partial^\beta[(\psi_\omega-P_{2,0} \psi_\omega)^\mu](0) \partial^{2\alpha-\beta} a_{k_1,k_2}(i y_\omega u_1)= O\left(  y_\omega^{k_2-\abs{2\alpha-\beta}}  \right) =O\left( y_\omega^{k_2- 2 j+1} \right)
\]
for $\omega\to 0$. {Otherwise, let $\abs{2\alpha-\beta}=2 j$, so that $\mu=0$ and $\beta=0$}. If $\alpha\neq j u_1$, then~\eqref{asympak1k2in0} implies that
\[
\partial^{2 \alpha} a_{k_1,k_2}(i y_\omega u_1)= O\left(y_\omega^{k_2- 2 j+2}  \right)=O\left( y_\omega^{k_2-2 j+1}  \right),
\]
while, if $\alpha=j u_1$,
\[
\partial_1^{2 j} a_{k_1,k_2}(i y_\omega u_1)=(-1)^{k_1+k_2} i^{-2 j} \frac{k_2 !}{(k_2- 2 j)!} y_\omega^{k_2-2 j}. 
\]
From this and the fact that \[\theta'(0)=\lim\limits_{\omega\to 0} \frac{\omega}{y_\omega}=\frac{2}{3}\]
we get the asserted estimate.
	 \end{proof}
{
Lemma~\ref{lemmaHomega0} above gives the expansions 2 and 3 of Theorem~\ref{HstimeI}. Indeed, it allows us to choose $k$ in Proposition~\ref{propHorm} as
\begin{itemize}
\item[2.] $k=k_2/2$ if $k_2$ is even, since in this case the last term of the sum in~\eqref{eq:HstimeI} is
\[\frac{c_{k_1,k_2,k_2/2}}{|x|^{k_2}} + O\left(\frac{\omega}{|x|^{k_2}}\right)\]
which is bigger than the remainder.
\item[3.] $k=(k_2-1)/2$ if $k_2$ is odd and $|t|\to \infty$, since in this case the last term of the sum in~\eqref{eq:HstimeI} is
\[c_{k_1,k_2,(k_2-1)/2}\frac{\omega}{|x|^{k_2-1}} + O\left(\frac{\omega^2}{|x|^{k_2-1}}\right) = c_{k_1,k_2,(k_2-1)/2}\frac{|t|}{|x|^{k_2+1}} + O\left(\frac{|t|^2}{|x|^{k_2+3}}\right)\]
which is bigger than the remainder, since $|t|\to \infty$.
\end{itemize}
}	

The case 4 of Theorem~\ref{HstimeI}, that is the case when $k_2$ is odd, $\omega \to 0$ and $|t|$ is bounded, has to be treated in a different way, since $\omega/|x|^{k_2-1}$ may be comparable {with the remainder} $1/|x|^{k_2+1}$ {or even smaller}. Thus, the development given above may not be sharp in this case. To overcome this difficulty, we make use of the following lemma. For the reader's convenience, we also consider $k_2$ even and a stronger statement than that we need (see Remark~\ref{rem:sharpness}).

\begin{lemma}\label{lemmaTomega0}
Let $N\in \N$. Then, when $\omega \to 0$,
\[
p_{1,k_1,k_2}(x,t)= \sum_{h=0}^N\frac{1}{(2h+1)!}|t|^{2h+1}p_{1,k_1,k_2+2h+1}(x,0) + O\left(\abs{t}^{2N+3} p_{1,k_1,k_2+2N+3}(x,0)\right)
\]
if $k_2$ odd; if $k_2$ is even
\[
p_{1,k_1,k_2}(x,t)= \sum_{h=0}^N\frac{1}{(2h)!}|t|^{2h}p_{1,k_1,k_2+2h}(x,0) + O\left(\abs{t}^{2N+2} p_{1,k_1,k_2+2N+2}(x,0)\right).
\]
\end{lemma}

\begin{proof}
Assume that $k_2$ is odd. Then
\begin{align*}
&(4\pi)^n(2\pi)^m\abs*{p_{1,k_1,k_2}(x,t) - \sum_{h=0}^N\frac{1}{(2h+1)!}|t|^{2h+1}p_{1,k_1,k_2+2h+1}(x,0) }
\\&= \abs*{\int_{\R^m} e^{-\frac{|x|^2}{4}|\lambda|\coth|\lambda|}\frac{\abs{\lambda}^{n+k_1}\cosh(|\lambda|)^{k_1}}{\sinh(|\lambda|)^{n+k_1}}(\lambda,u_1)^{k_2}\left\{ e^{i|t|(\lambda,u_1)} -\sum_{h=0}^N\frac{\left[i|t|(\lambda,u_1)\right]^{2h+1}}{(2h+1)!}\right\}\,\dd \lambda}
\\& \leq \frac{|t|^{2N+3}}{(2 N+3)!}\int_{\R^m} e^{-\frac{|x|^2}{4}|\lambda|\coth|\lambda|}\frac{\abs{\lambda}^{n+k_1}\cosh(|\lambda|)^{k_1}}{\sinh(|\lambda|)^{n+k_1}}(\lambda,u_1)^{k_2+2N+3}\,\dd \lambda\\&= \frac{(4\pi)^n(2\pi)^m }{(2 N+3)!} \abs{t}^{2N+3} \abs*{p_{1,k_1,k_2+2N+3}(x,0)}.
\end{align*}
The first assertion is then proved. The proof in the case $k_2$ even is analogous.
 \end{proof}
Thus, the case $\omega\to 0$ while $|t|$ remains bounded when $k_2$ is odd can be related to the same case when $k_2$ is even, which is completely described by Lemma~\ref{lemmaHomega0}. {Observe that the expansion appearing in Theorem~\ref{HstimeI}, 4, is obtained with the choice $N=0$ in Lemma~\ref{lemmaTomega0}.}

We finally consider the case $\omega \to \frac{\pi}{2}$, which as above provides the expansions 5 and 6 of Theorem~\ref{HstimeI}.
	\begin{lemma}\label{lemmaHomegapi2}
Define, for $j\in \N$ such that $2j\leq k_1$,
\begin{equation}\label{bk1k2j}
b_{k_1,k_2,j}\coloneqq 
		(-1)^{k_2} \frac{k_1 !}{2^{k_1-2j}(k_1-2 j)!  j!}\ \left(\frac{\pi}{2}\right)^{n+k_1+k_2},
\end{equation}	
and, when $k_1$ is odd,
\begin{equation}\label{bk1k2odd}
b_{k_1,k_2,(k_1+1)/2}\coloneqq(-1)^{k_2}\frac{(k_1+1)!}{[(k_1+1)/2]!}\left(\frac{\pi}{2}\right)^{n+k_1+k_2-1}\left( n+k_1+k_2+\frac{\pi^2}{24}(k_1+2) +\frac{3}{2}(m-1)  \right).
\end{equation}
		Then, for $\omega\to \frac{\pi}{2}$, if $2j \leq k_1$
		\[
4^j L_{j,\psi_\omega} a_{k_1,k_2,\omega}= b_{k_1,k_2,j} \left(\omega-\frac{\pi}{2}\right)^{k_1-2 j} + O\left(\left(\omega-\frac{\pi}{2}\right)^{k_1-2 j+1}  \right)
		\]
		while if $k_1$ is odd, then
	\[2^{k_1+1}	L_{(k_1+1)/2, \psi_\omega}a_{k_1,k_2,\omega}=
b_{k_1,k_2,(k_1+1)/2} + O\left(\omega-\frac{\pi}{2}\right).\]
	\end{lemma}

	\begin{proof}
By elementary computations,
	\begin{align}\label{ak1k2pi2}
		&a_{k_1,k_2,\pi/2}\left(\lambda\right)=(-1)^{k_1} i^{k_2-n}\left(i\frac{\pi}{2}\right)^{n+k_1+k_2} \lambda_1^{k_1}\nonumber\\ 
			&\quad+ (-1)^{k_1} i^{k_2-n}\left(i\frac{\pi}{2}\right)^{n+k_1+k_2-1}\left((n+k_1+k_2) \lambda_1^{k_1+1} +\frac{k_1}{2} \lambda_1^{k_1-1} (\lambda^2-\lambda_1^2)\right)+O\left(\abs{\lambda}^{k_1+2} \right).
	\end{align}
Therefore, since $a_{k_1,k_2,\pi/2}$ is analytic on its domain, we infer that,  for every $h=0,\dots, k_1$ we have
		\begin{equation}\label{ak1k2pi2bis}
		a_{k_1,k_2,\pi/2}^{(h)}(\lambda)=(-1)^{k_1} i^{k_2-n} \left(i\frac{\pi}{2}  \right)^{n+k_1+k_2} \frac{k_1!}{(k_1-h)!}  \lambda_1^{k_1-h} u_1^{\otimes h} +O\left(\abs{\lambda}^{k_1-h+1}  \right) 
		\end{equation}
	as $\lambda\to 0$.

\smallskip
	
Consider first $j$ such that $2j\leq k_1$. Then, arguing as in the proof of Lemma~\ref{lemmaHomega0} and taking into account~\eqref{ak1k2pi2bis} and the fact that
		\[
y_\omega-\frac{\pi}{2}= \frac{1}{2}\left(\omega-\frac{\pi}{2}\right)+ O\left[\left(\omega-\frac{\pi}{2}\right)^2\right]
		\]
		when $\omega \to \pi/2$, the first assertion follows.

\smallskip

		Let now $k_1$ be odd, so that $(k_1+1)/2$ is an integer. We shall prove that \[2^{k_1+1} L_{(k_1+1)/2, \psi_{\pi/2}}a_{k_1,k_2,\pi/2}=b_{k_1,k_2,(k_1+1)/2}.\] The estimate in the statement is then a consequence of this by Taylor expansion.
		
{Since $(\psi_{\pi/2}''(0)^{-1}\partial, \partial )^{\mu+(k_1+1)/2}$ is a differential operator of degree $2\mu +k_1+1$ while $ [ (\psi_\omega-P_{2,0}\psi_\omega)^\mu a_{k_1,k_2,\omega} ]$ is infinitesimal of degree $3\mu +k_1$ at $0$, the only terms in the sum~\eqref{Lj} (with $j=(k_1+1)/2$) which are not zero are clearly those for which} \[2\mu +k_1+1 \geq 3\mu+k_1,\]
namely $\mu \leq 1$. Consider first $\mu=0$. Then, since
$\theta'(y_{\pi/2})=2$, by~\eqref{psi2partialpartial}
		\[
		\begin{split}
		(\psi_{\pi/2}''(0)^{-1}\partial, \partial)^{(k_1+1)/2}a_{k_1,k_2,\pi/2}(0)&=i^{-(k_1+1)/2}\sum_{\abs{\alpha}=(k_1+1)/2} \frac{[(k_1+1)/2]! }{2^{\alpha_1}\alpha!} \partial^{2 \alpha} a_{k_1,k_2,\pi/2}(0).
		\end{split}
		\]
	Observe that, by~\eqref{ak1k2pi2}, $\partial^{2\alpha} a_{k_1,k_2,\pi/2}(0)\neq 0$ only if $\alpha=((k_1-1)/2)u_1+u_h$ for some $h=1,\dots, m$. {For the choice $h=1$},
		\[
		\partial^{k_1+1}_1 a_{k_1,k_2,\pi/2}(0)=(-1)^{k_1} i^{k_2-n} \left( i\frac{\pi}{2} \right)^{n+k_1+k_2-1} (k_1+1)! (n+k_1+k_2)
		\]
		while, for $h=2,\dots, m$,
		\[
		\partial^{k_1-1}_1 \partial_h^2 a_{k_1,k_2,\pi/2}(0)=(-1)^{k_1}{i^{k_2-n}  \left( i\frac{\pi}{2} \right)^{n+k_1+k_2-1} k_1!}
		\]
		so that
\begin{align*}
		(\psi_{\pi/2}''(0)^{-1}\partial, \partial)^{(k_1+1)/2}& a_{k_1,k_2,\pi/2}(0)\\&= (-1)^{k_1}\frac{i^{k_2-n-\frac{k_1+1}{2}}}{2^{\frac{k_1+1}{2}}}\left(i\frac{\pi}{2}\right)^{n+k_1+k_2-1}(k_1+1)! (n+k_1+k_2+m-1).
\end{align*}
Consider now $\mu=1$. Then by~\eqref{psi2partialpartial}
		\begin{multline*}
			(\psi_{\pi/2}''(0)^{-1}\partial, \partial)^{(k_1+3)/2}\left[ (\psi_{\pi/2}-P_{2,0} \psi_{\pi/2}) a_{k_1,k_2,\pi/2} \right](0)\\=i^{-(k_1+3)/2}\sum_{\abs{\alpha}=(k_1+3)/2} \frac{[(k_1+3)/2]! }{2^{\alpha_1}\alpha!} \partial^{2 \alpha} \left[(\psi_{\pi/2}-P_{2,0}\psi_{\pi/2})a_{k_1,k_2,\pi/2} \right](0).
\end{multline*}
Since
		\[
		\psi_{\pi/2}'''(0)= 
		 \pi u_1 \otimes u_1 \otimes u_1 + \frac{2}{\pi} \sum_{h=2}^m (u_1 \otimes u_h\otimes u_h+  u_h\otimes u_1 \otimes u_h + u_h\otimes u_h \otimes u_1 ),
		\]
we deduce that the only $\alpha$ for which we get a non-zero term in the above sum are $u_1(k_1+1)/2+ u_h$ for $h=1,\dots, m$.
		Now,
		\[
		\partial_1^{k_1+3} \left[ (\psi_{\pi/2}-P_{2,0} \psi_{\pi/2}) a_{k_1,k_2,\pi/2}\right](0)= \frac{(k_1+3)!}{ 3 !}  (-1)^{k_1} i^{k_2-n} \pi\left(i \frac{\pi}{2}\right)^{n+k_1+k_2}   ,
		\]
		while, for $h=2,\dots, m$, 
		\[
		\partial_1^{k_1+1} \partial_h^2 \left[ (\psi_{\pi/2}-P_{2,0} \psi_{\pi/2})a_{k_1,k_2,\pi/2} \right](0)=	\frac{2}{\pi}(-1)^{k_1}{i^{k_2-n} \left(i\frac{\pi}{2}\right)^{n+k_1+k_2}(k_1+1)!  		}.		
		\]
		Therefore,
\begin{align*}
		(\psi_{\pi/2}''(0)^{-1}&\partial, \partial)^{(k_1+3)/2}\left[ (\psi_{\pi/2}-P_{2,0} \psi_{\pi/2})a_{k_1,k_2,\pi/2} \right](0)\\&= (-1)^{k_1}i^{k_2-\frac{k_1+1}{2}}\frac{(k_1+1)!}{2^{(k_1+3)/2}}i^{-n}	 \left(i \frac{\pi}{2}\right)^{n+k_1+k_2-1}(k_1+3) \left[ \frac{\pi^2}{12}(k_1+2) + m-1	\right]	
\end{align*}
		from which one gets the asserted estimate. 
	 \end{proof}
Theorem~\ref{HstimeI} is now completely proved. In the following table we summarize the asymptotic behaviour, without remainders, of $\Upsilon(x,t)$. 

\begin{minipage}{\textwidth}
\begin{center}
\footnotesize{
\captionof*{table}{Asymptotic behaviour of $\Upsilon(x,t)$ in case \textbf{I}: principal part}	
$\begin{array}{*{5}{|c}}

\hline

\begin{array}{c}\varepsilon\leq \omega \leq \pi/2-\varepsilon \\ \text{or} \\ \pi/2 +\varepsilon\leq \omega\leq C\end{array} & \multicolumn{3}{c|}{ \displaystyle (-1)^{k_1+k_2}  \frac{ y_\omega^{n+k_1+k_2}\cos(y_\omega)^{k_1}}{\sin (y_\omega)^{n+k_1}}}  \\ 
   
    \hline
\multirow{3}{*}{\begin{tabular}[c]{@{}c@{}}\\ \; \\ \; \\  $\omega \to 0$\\ \end{tabular}}               &\multicolumn{2}{c|}{\begin{array}{c} \\ k_2  \text{ even}\\ \; \end{array} }
& \displaystyle \sum_{j=0}^{k_2/2}c_{k_1,k_2,j} \frac{\omega^{k_2-2j}}{|x|^{2j}} \\ \cline{2-4}
&
 
\multirow{2}{*}{\begin{tabular}[c]{@{}c@{}}\\  \; \parbox{0.09\textwidth}{\begin{center}$k_2$ odd \;\end{center}} \\ \end{tabular}}                             & \begin{array}{c} \\ |t| \to \infty\\ \; \end{array}  & \displaystyle\sum_{j=0}^{(k_2-1)/2} c_{k_1,k_2,j}\frac{ \omega^{k_2-2j}}{|x|^{2j}}\\ \cline{3-4} 
                                              &                                                        & \begin{array}{c} \\ 0\leq |t|\leq C\\ \; \end{array}                             & \displaystyle c_{k_1,k_2+1, (k_2+1)/2}\frac{|t|}{|x|^{k_2+1}} \\ \hline

\multirow{2}{*}{\begin{tabular}[c]{@{}c@{}}\\ \parbox{0.2\textwidth}{\begin{center}$\omega \to \frac{\pi}{2}$\end{center}}\\ \end{tabular}}              
& \multicolumn{2}{c|}{\begin{array}{c} \\ k_1 \;\text{even} \\ \; \end{array} }  & \displaystyle \sum_{j=0}^{k_1/2} b_{k_1,k_2,j} \frac{\left(\omega -\frac{\pi}{2}\right)^{k_1-2j}}{|x|^{2j}} \\ \cline{2-4} 

 & \multicolumn{2}{c|}{\begin{array}{c} \\ k_1 \; \text{odd} \\ \; \end{array} }  & \; \quad  \displaystyle \sum_{j=0}^{(k_1-1)/2} b_{k_1,k_2,j} \frac{\left(\omega -\frac{\pi}{2}\right)^{k_1-2j}}{|x|^{2j}} + \frac{b_{k_1,k_2,(k_1+1)/2}}{|x|^{k_1+1}}\quad \;\\ \hline
 
\end{array}
$}
\end{center}
\end{minipage}
	\subsection*{The Other Cases}
	We now consider the case $\omega \to +\infty$. We begin by showing that, when $m$ is odd, matters can be reduced to the case $m=1$.
	\begin{lemma}\label{dispariesattolemma}
		When $m$ is odd, $m\geq 3$,
		\begin{equation}\label{dispariesatto}
			\begin{split} p_{1,k_1,k_2}^{(m)}(x,t)= \sum_{k=1}^{\frac{m-1}{2}}\frac{c_{m,k}(-1)^k}{(2\pi)^{\frac{m-1}{2}}} \sum_{r=0}^{k_2} \binom{k_2}{r} \frac{(-1)^r (m-1-k)_r}{|t|^{m-1-k+r}}p_{1,k_1,k_2+k-r}^{(1)}(x,|t|), \end{split}
		\end{equation}
		where \[c_{m,k}= \frac{(m-k-2)!}{2^{\frac{m-1}{2}-k}\left( {\frac{m-1}{2}} -k\right)! (k-1)!  }\]
		and $(m-1-k)_r = (m-1-k) \cdots (m-1-k+r-1)$ is the Pochhammer symbol\footnote{See, e.g.,~\cite{Erdelyi2}.}.
	\end{lemma}
	\begin{proof}
		
		Let $m$ be odd, $m\geq 3$. We first pass to polar coordinates in \eqref{pk1k2} for $k_2=0$, and get
		\[
		p_{1,k_1,0}^{(m)}(x,t)= \frac{(-1)^{\frac{m-1}{2}}}{(2\pi)^m (4\pi)^n}\int_{0}^\infty \int_{S^{m-1}} e^{i\rho\abs{t}( \sigma, u_1)}\,\dd \sigma \, e^{-R\rho \coth(\rho)}a_{k_1,m-1}(\rho)\,\dd \rho
		\]
		where $\dd \sigma$ is the $(m-1)$-dimensional (Hausdorff) measure on $S^{m-1}$ and $a_{k_1,m-1}$ is the function defined in~\eqref{ak1k2}.
		Since the Bessel function is an elementary function when $m$ is odd, one can prove that (see e.g.~\cite[equation (6.5)]{Eldredge} and references therein)\footnote{This is why we had to restrict to the case $k_2=0$;  otherwise, we would get the additional term $( \sigma, u_1)^{k_2}$ in the integral on the sphere.
		}
		\begin{equation*}\label{Eldredge:elem}
			\int_{S^{m-1}} e^{i\rho\abs{t}( \sigma, u_1)}\,\dd \sigma= 2(2\pi)^{\frac{m-1}{2}}\Re \frac{e^{i\rho\abs{t} }}{(\rho\abs{t})^{m-1}}\sum_{k=1}^{\frac{m-1}{2}} c_{m,k}(-i\abs{t} \rho)^k.
		\end{equation*}
		This yields
		\[
		\begin{split}
		p^{(m)}_{1,k_1,0}(x,t)&= \sum_{k=1}^{\frac{m-1}{2}}\frac{c_{m,k}(-1)^k}{(2\pi)^{\frac{m-1}{2}}} \frac{1}{|t|^{m-1-k}} p^{(1)}_{1,k_1,k}(x,|t|)
		\end{split}
		\]
		which gives \eqref{dispariesatto}, since $p_{1,k_1,k_2}^{(m)}(x,t)= \frac{\partial^{k_2}}{\partial |t|^{k_2}}p_{1,k_1,0}^{(m)}(x,t)$ by definition.
	 \end{proof}
	
	\begin{corollary} \label{stime234dispari}
		{Let $m$ be odd. Then, when $(x,t)\to \infty$ and $\delta\to 0^+$}
		\begin{equation}\label{cor:odd}
			\begin{split} p_{1,k_1,k_2}^{(m)}(x,t)= \frac{(-1)^{k_2} \pi^{k_1+k_2}}{2^{n-k_1+1 + \frac{m-1}{2}}}|t|^{n+k_1-1-\frac{m-1}{2}}e^{-\frac{1}{4}d(x,t)^2 } e^{-\kappa \rho(\delta)} \tilde I_{n+k_1-1}(\kappa \rho(\delta))\left[1+ g(\abs{x},|t|) \right],\end{split}
		\end{equation}
		where $g$ satisfies the estimates~\eqref{remainders}.
	\end{corollary}
	\begin{proof}
		If $m=1$, the statement reduces to Corollary~\ref{cor:1}. Suppose then $m\geq 3$. Since $p^{(1)}_{1,k_1,r}\asymp p^{(1)}_{1,k_1,k_2}$ for every $0\leq r\leq k_2$ by Corollary~\ref{cor:1}, the principal term in~\eqref{dispariesatto} corresponds to $r=0$, $k=\frac{m-1}{2}$. Hence
		\begin{equation}\label{eqdispariprinc}
			p_{1,k_1,k_2}^{(m)}(x,t)=\frac{(-1)^{\frac{m-1}{2}}}{(2\pi)^{\frac{m-1}{2}}} \abs{t}^{-\frac{m-1}{2}}p_{1,k_1,k_2+\frac{m-1}{2}}^{(1)}(x,t)\left[1+O\left(\frac{1}{\abs{t}}\right)\right].
		\end{equation}
		Now substitute the estimate given by Corollary~\ref{cor:1} into~\eqref{eqdispariprinc}. The remainder $g$ in~\eqref{cor:odd} still satisfies~\eqref{remainders}, since~\eqref{remainders} is satisfied by $1/\abs{t}$.
	 \end{proof}
Let now $m$ be even, $m\geq 2$. We start by a descent method, in the same spirit of~\cite{Eldredge}: indeed, observe that the Fourier inversion formula yields
	\[
	p^{(m)}_{1,k_1,0}(x,t)=\int_\R p^{(m+1)}_{1,k_1,0}(x,(t,t_{m+1}))\,\dd t_{m+1},
	\]
	so that, by differentiating under the integral sign,
	\[
	p^{(m)}_{1,k_1,k_2}(x,t)=\int_\R \frac{\partial^{k_2}}{\partial \abs{t}^{k_2}}p^{(m+1)}_{1,k_1,0}(x,(t,t_{m+1}))\,\dd t_{m+1}.
	\]
	Observe that $\abs{(t,t_{m+1})} = \abs{t}\sqrt{1+ \frac{t_{m+1}^2}{\abs{t}^2}}$. Therefore, if we define $\mathfrak{I}^{k_2}\coloneqq\{h\in \N^{k_2}\colon \sum_{j=1}^{k_2} j h_j=k_2\}$, Faà di Bruno's formula applied twice\footnote{Applied once, it yields
		\[
		\frac{\partial^{k_2}}{\partial \abs{t}^{k_2}}p^{(m+1)}_{1,k_1,0}(x,(t,t_{m+1}))=\sum_{h\in \mathfrak{I}^{k_2}}\frac{k_2 !}{h !} p_{1,k_1,\abs{h}}^{(m+1)}(x,(t,t_{m+1})) \prod_{j=1}^{k_2} \left(  \frac{1}{j !}\frac{\partial^{j}}{\partial \abs{t}^j}\sqrt{\abs{t}^2+t_{m+1}^2}  \right) ^{h_j},
		\]
		and then
		\[
		\frac{\partial^{j}}{\partial \abs{t}^j}\sqrt{\abs{t}^2+t_{m+1}^2}= \sum_{ \ell_1+2\ell_2=j} \frac{j !}{\ell !} {(-1)^{\abs{\ell}}} \left({-\frac{1}{2}}\right)_{\abs{\ell}} (\abs{t}^2+t_{m+1}^2)^{\frac{1}{2}-\abs{\ell}} (2\abs{t})^{\ell_1} .
		\]
	} leads to
	
	\begin{equation*} 
		p_{1,k_1,k_2}^{(m)}(x,t) = \sum_{h\in \mathfrak{I}^{k_2}} \frac{k_2!}{h!} \int_\R p_{1,k_1,|h|}^{(m+1)}(x,(t,t_{m+1}))\,F_{h}(t,t_{m+1})\,\dd t_{m+1}
	\end{equation*}	
	where
	\[ \begin{split} F_{h}(t,t_{m+1})&= \prod_{j=1}^{k_2} \left( \sum_{\ell_1 +2\ell_2 = j} \frac{2^{\ell_1}}{\ell!} {(-1)^{\abs{\ell}}} \left({-\frac{1}{2}}\right)_{|\ell|}|t|^{1-j}\left(1+\frac{t_{m+1}^2}{|t|^2}\right)^{\frac{1}{2} -|\ell|}\right)^{h_j}.
	\end{split} \]	
	{Since $F_{(k_2,0,\dots,0)}= \left(1+\frac{t_{m+1}^2}{|t|^2}\right)^{-k_2/2} $ while $F_h=O\left(\frac{1}{|t|}\left(1+\frac{t_{m+1}^2}{|t|^2}\right)^{-1/2}\right)$ otherwise, we have proved the following lemma.}
	\begin{lemma}\label{lemmaalpha} When $m$ is even, $m\geq 2$, 
		\begin{multline*}
			p_{1,k_1,k_2}^{(m)}(x,t) =\int_\R \left(1+\frac{t_{m+1}^2}{|t|^2}\right)^{-\frac{k_2}{2}} p_{1,k_1,k_2}^{(m+1)}(x,(t,t_{m+1})) \,\dd t_{m+1} \\ +  O \left[\frac{1}{|t|}\max_{0\leq r< k_2}\int_\R \left(1+\frac{t_{m+1}^2}{\abs{t}^2}\right)^{-\frac{1}{2}} p_{1,k_1,r}^{(m+1)}(x,(t,t_{m+1})) \,\dd t_{m+1}\right]. \end{multline*}
	\end{lemma}
	As a consequence of Lemma \ref{lemmaalpha}, {matters can be reduced to finding} the asymptotic expansions of the integrals
	\begin{equation}\label{alphar} \int_\R  \left(1+\frac{t_{m+1}^2}{\abs{t}^2}\right)^{\alpha} p_{1,k_1,r}^{(m+1)}(x,(t,t_{m+1}))\,d t_{m+1}
	\end{equation}
	when $\alpha \in \R$ and $0\leq r\leq k_2$. From these, it will also be proved that the remainder in Lemma \ref{lemmaalpha} is indeed smaller than the principal part, which \emph{a priori} is not obvious.
	
	\smallskip
	
	With this aim, we define the function $\sigma\colon \R\ni s \mapsto \sqrt{1+s^2}$, and write $t'=(t,t_{m+1}) \in \R^{m+1}$. It is straightforward to check that $\abs{t'}=\abs{t}\sigma\left(\frac{t_{m+1}}{\abs{t}}\right)$. Thus, define 
	\[\delta(s) \coloneqq \frac{\delta}{\sqrt{\sigma(s)}},\qquad \kappa(s)\coloneqq \kappa \sqrt{\sigma(s)} = 2\pi|t|\delta\sqrt{\sigma(s)}.\]
	Obviously, $\delta(0)=\delta$ and $\kappa(0)=\kappa$. If we put a prime on the quantities introduced in Definition~\ref{Romegadeltakappa} relative to $t'$, moreover,
	\begin{gather*}
		\delta'=\delta\left(\frac{t_{m+1}}{|t|}\right), \qquad
		\kappa'=\kappa\left(\frac{t_{m+1}}{\abs{t}}\right).
	\end{gather*}
	In cases {\bf II}, {\bf III} and {\bf IV}, $|t| \to \infty$ and $\delta \to 0^+$.  By substituting~\eqref{cor:odd} into~\eqref{alphar} and by the change of variable $\frac{t_{m+1}}{|t|} \mapsto s$ in the integral
	\begin{equation*}  \eqref{alphar}=\frac{(-1)^{r} \pi^{r+k_1}}{2^{n-k_1+1 + \frac{m}{2}}}|t|^{n+k_1-1 -\frac{m}{2}+1} e^{-\frac{1}{4}d(x,t)^2 } e^{-\kappa \rho(\delta)} \mathscr{I}_{2\alpha +n+k_1-1 -\frac{m}{2}}, 
	\end{equation*}
	where
	\begin{equation}\label{Ibeta}
		\mathscr{I}_{\beta}= \int_{\R}\sigma(s)^{\beta} e^{-|t|\pi(\sigma(s) -1)}\tilde I_{n+k_1-1}\left(\kappa(s)\rho\left(\delta(s)\right) \right)\left[ 1+g(\abs{x},|t|\sigma(s))\right]\,\dd s,
	\end{equation}
and $g$ satisfies the estimates~\eqref{remainders}. Therefore, matters can be reduced to finding some asymptotic estimates of the integrals $\mathscr{I}_{\beta}$.
	
	\subsection{\textbf{II.} Estimates for $\delta \rightarrow 0^+$ and $\kappa \rightarrow +\infty$.} 
	\begin{theorem}\label{HstimeII} For $\delta\to 0^+$ and $\kappa\to +\infty$
		\[
		\begin{split} 
		p_{1,k_1,k_2}(x,t)= \frac{(-1)^{k_2} \pi^{k_1+k_2}}{ 4^n (\pi\delta)^{n+k_1-\frac{m+1}{2}} \sqrt{2 \pi\kappa^m} }e^{-\frac{1}{4}d(x,t)^2}\left[1+O\left(\delta+\frac{1}{\kappa} \right) \right].
		\end{split}
		\]	
	\end{theorem}
	\begin{proof}
		{When $m$ is odd, the theorem is obtained by combining Theorem \ref{stimeII} with \eqref{eqdispariprinc}}. Therefore, we only consider $m$ even. By the preceding arguments, it will be sufficient to study $\mathscr{I}_\beta$ in \eqref{Ibeta}.
		
		Since the argument of the modified Bessel function tends to $+\infty$, we use the development \eqref{modibess}, which gives
		\begin{multline*} 
			\mathscr{I}_{\beta}=\frac{(2\pi)^{-n-k_1} e^{\kappa\rho(\delta)}}{\delta^{n+k_1-\frac{1}{2}}|t|^{n+k_1-\frac{1}{2}}} \int_\R e^{-|t|\phi_\delta(s)}\frac{\sigma(s)^{\beta - \frac{1}{4} - \frac{n+k_1-1}{2}}}{\rho\left(\delta(s)\right)^{n+k_1-\frac{1}{2}}}\\ \times \left[1+ O\left(\frac{1}{\delta |t|\sqrt{\sigma(s)}}\right)\right] \left[ 1+g(\abs{x},|t|\sigma(s))\right]\,\dd s
		\end{multline*}
		where
		\[
		\phi_\delta(s)= \pi [\sigma(s)-1]+ 2\pi \delta\left[\rho(\delta) - \sqrt{\sigma(s)}\rho\left(\delta(s)\right) \right].
		\]
		We first study the principal part of the integral, to which we apply Laplace's method (see Remark \ref{Laplace}) with \[\mathscr{F}=\{\phi_\delta\colon \delta \in [0,\delta_2]\},\qquad \mathscr{G}= \left\{\frac{\sigma(\cdot)^{\beta - \frac{1}{4} - \frac{n+k_1-1}{2}}}{\rho\left(\delta(\cdot)\right)^{n+k_1-\frac{1}{2}}}\colon \delta \in [0,\delta_2]\right\}\]
for some $\delta_2$, smaller than the $\delta_1$ of Lemma~\ref{stime:lem:6}, to be determined.
		\begin{itemize}
			\item[2.] It is easily seen that $\phi_\delta(0)=0$. Moreover
			\begin{equation}\label{phiprimodelta}
				\phi_\delta'(s)= \pi \frac{s}{\sigma(s)}\left[ 1- \delta(s) \rho\left(\delta(s)\right)+ \frac{\delta^2}{\sigma(s)^{\frac{3}{2}}} \rho'\left( \delta(s)\right)\right],
			\end{equation}
			so that $\phi_\delta'(0)=0$ and $\phi_\delta''(0)= \pi( 1-\delta\rho(\delta)+\delta^2 \rho'(\delta) )$. Observe that there is $\delta_2>0$, which we may choose smaller than $\delta_1$, such that 
			\begin{equation}\label{1menocose}
				1- \delta(s) \rho\left(\delta(s)\right)+ \frac{\delta^2}{\sigma(s)^{\frac{3}{2}}} \rho'\left( \delta(s)\right)\geq \frac{1}{2}
			\end{equation}
			for every $s$ and every $\delta \in [0,\delta_2]$. Therefore, $\phi''_\delta(0)\geq \frac{\pi}{2}$ for every $\delta\in [0,\delta_2]$.
			
			\item[3.] {By~\eqref{phiprimodelta} and~\eqref{1menocose}, for $s\in\R$ and $\delta \in (0,\delta_2)$},
			\begin{equation}\label{absphidelta}
				\abs{\phi_\delta'(s)} \geq \frac{\pi}{2\sigma(s)} \abs{s}.
			\end{equation}
			In particular, $\abs{\phi_\delta'(s)}\geq \frac{\pi}{2\sigma(2)} \abs{s}$ for every $s\in [-2,2]$. 
			
			\item[1.] Observe that $\phi_\delta'(s)= \sign(s)\abs{\phi'_\delta(s)}$ by~\eqref{phiprimodelta}; then, by~\eqref{absphidelta},
			\[
			\phi_\delta(s)=\int_0^s \sign(s)\abs{\phi'_\delta(u)}\,\dd u=\abs*{\int_0^s \abs{\phi'_\delta(u)}\,\dd u}\geq \frac{\pi}{2\sigma(s)} \abs*{\int_0^s \abs{u}\,\dd u}\geq  \frac{\pi s^2}{4\sigma(s)}
			\]
			for every $s\in\R$, since $\sigma$ is even and increasing on $[0,\infty)$.
			\item[4.] By definition of $\sigma$ and since $\rho$ is continuous in zero, we get $g(s) \lesssim |s|^{\beta - \frac{1}{4} - \frac{n+k_1-1}{2}}$ for $s\to \infty$, uniformly in $g \in \mathscr{G}$.
		\end{itemize}
		
		{By Theorem~\ref{stationaryphase}, then,}
		\[
		\begin{split}
		\int_\R e^{-\abs{t} \phi_\delta(s)}\frac{\sigma(s)^{\beta - \frac{1}{4} - \frac{n+k_1-1}{2}}}{\rho\left(\delta(s)\right)^{n+k_1-\frac{1}{2}}} \,\dd s&= \sqrt{\frac{2}{\abs{t}\left( 1-\delta\rho(\delta)+\delta^2 \rho'(\delta)\right) }} \left[1+O\left( \frac{1}{\abs{t}}\right)  \right] \\
		&=\sqrt{\frac{2}{\abs{t}}} \left[1+O\left(\delta+\frac{1}{\abs{t}}\right)\right].
		\end{split}
		\]
		The remainder can be treated similarly, and with the same arguments as above one gets
		\[\begin{split}
		\int_\R e^{-|t|\phi_\delta(s)}&\frac{\sigma(s)^{\beta - \frac{1}{4} - \frac{n+k_1-1}{2}}}{\rho\left(\delta(s)\right)^{n+k_1-\frac{1}{2}}} {\left[O\left(\frac{1}{\delta |t|\sqrt{\sigma(s)}}\right) + O\left(\frac{1}{\kappa}+\delta \right)\right]}\,\dd s\\
		&\qquad\qquad= \sqrt{\frac{2}{\abs{t}}}
		{\left[ 1+ O\left(\delta + \frac{1}{\abs{t}}\right)  \right] O\left(\frac{1}{\delta \abs{t}}+\frac{1}{\kappa}+\delta  \right)}=\sqrt{\frac{2}{\abs{t} }}O\left(\frac{1}{\kappa} +\delta\right)
		\end{split}\]
		since {$\frac{1}{\delta |t|}= \frac{2\pi }{\kappa}=O\left(\frac{1}{\kappa}\right)$} and $1/\sqrt{\sigma(s)}\leq 1$ for every $s\in \R$. The proof is complete.
	 \end{proof}
	\subsection{\textbf{III} and \textbf{IV}. Estimates for $\delta\to 0^+$ and $\kappa$ bounded.}
	These two cases can be treated together and the principal part of $p_{1,k_1,k_2}^{(m)}$ is easy to get. The remainders are more tricky, since when passing from the $m$-dimensional variable $t$ to the $(m+1)$-dimensional variable $t'$ the {asymptotic conditions in} {\bf II}, {\bf III} and {\bf IV} do not correspond to {those in} {\bf II'}, {\bf III'}, {\bf IV'} (these symbols standing for the cases relative to $m+1$); on the contrary, they mix together according to the values of the additional variable $t_{m+1}$.
	\begin{theorem}\label{HstimeIIIeIV} Fix $C>1$. If $\delta\to 0^+$ while {$1/C\leq \kappa \leq C$}, then
		\begin{equation*}
			p_{1,k_1,k_2}(x,t)=  \frac{(-1)^{k_2} \pi^{k_1+k_2}}{ 4^n (\pi\delta)^{n+k_1-\frac{m+1}{2}} \kappa^{\frac{m-1}{2}} }e^{-\frac{1}{4}d(x,t)^2 } e^{-\kappa} I_{n+k_1-1}(\kappa )  \left[1+O\left(\delta \right) \right]	.
		\end{equation*}
	When $\kappa\to 0^+$ and $\abs{t}\to +\infty$
		\[
		\begin{split} 
		p_{1,k_1,k_2}(x,t)= \frac{(-1)^{k_2} \pi^{k_1+k_2}}{{2^{2n + \frac{m-1}{2}}}(n+k_1-1)!}\abs{t}^{n+k_1-1-\frac{m-1}{2}}e^{-\frac{1}{4}d(x,t)^2} \left[1+O\left(\kappa+\frac{1}{\abs{t}} \right) \right].
		\end{split}
		\]	
	\end{theorem}
	
	\begin{proof}
		{The theorem holds when $m$ is odd by Theorem~\ref{stimeIIIeIV} combined with~\eqref{eqdispariprinc}}.
		When $m$ is even, we shall apply Laplace's method to $\mathscr{I}_\beta$. We first deal with the principal part. Define first
		\[\phi(s)= \pi  \sigma(s) -\pi ,\]
		so that Theorem~\ref{stationaryphase} will be applied to 
		\[\mathscr{F}=\{\phi\}, \qquad \mathscr{G}= \{\sigma(\cdot)^{\beta} \tilde I_{n+k_1-1}\left(\kappa(\cdot)\rho\left(\delta(\cdot)\right) \right)\colon \delta\in [0,\delta_1), \,\kappa \in [0,C]\}\] where $\delta_1$ is that of Lemma~\ref{stime:lem:6}.
		\begin{itemize}
			\item[2.] Notice that $\phi(0)=0$, that $\phi'(s)= \pi \frac{s}{\sigma(s)}$, and that $\phi''(0)=\pi $.
			\item[1.] Observe that $\phi(s)=\pi \frac{s^2}{1+\sqrt{1+s^2}}\geq \pi\frac{s^2}{2+\abs{s}}$, for every $s\in \R$.
			
			\item[3.] It is easily seen that $\abs{\phi'(s)}\geq \frac{\pi}{\sigma(1)} \abs{s}$ for every $s\in [-1,1]$.
			\item[4.] Recall that by~\eqref{modibess}
			\[\tilde I_{n+k_1-1}(\kappa(s)\rho(\delta(s))) \lesssim e^{\kappa(s)\rho(\delta(s))}\lesssim e^{\kappa \sqrt{\sigma(s)}}\] as $s\to \infty$, uniformly as $\kappa\in [0, C]$ and $\delta\in [0,\delta_1)$. Hence, there is a constant  $c_1>0$ such that $\abs{\sigma(s)^{\beta}\tilde I_{n+k_1-1}\left(\kappa(s)\rho\left(\delta(s)\right) \right)}\leq c_1 e^{c_1\abs{s}} $.
		\end{itemize}
		Therefore, by Theorem~\ref{stationaryphase}
		\begin{equation*}
			\int_\R e^{-\abs{t}\phi(s) }\sigma(s)^\beta \tilde I_{n+k_1-1}\left(\kappa(s)\rho\left(\delta(s) \right) \right) \,\dd s=\sqrt{\frac{2}{\abs{t}}} \tilde I_{n+k_1-1}(\kappa \rho(\delta))\left[1+O\left(\frac{1}{\abs{t}} \right)  \right]
		\end{equation*}
		uniformly in $\kappa$ and $\delta$. Since $\tilde I_{n+k_1-1}(\kappa \rho(\delta))-\tilde I_{n+k_1-1}(\kappa)=O(\kappa \rho(\delta)-\kappa)=O(\delta^2)$ uniformly as $ \kappa \in[0, C]$ by Taylor's formula, we are done with the principal part. We now deal with the remainders, {namely}
		\[
		\mathscr{I}'_\beta =	\int_\R e^{-\abs{t}\phi(s) } \sigma(s)^\beta \tilde I_{n+k_1-1}\left(\kappa(s) \rho\left(\delta(s) \right) \right) g( \abs{x}, \abs{t}\sigma(s))\,\dd s 
		\]
		where
		\[
		g( \abs{x}, \abs{t}\sigma(s))= \begin{cases}
		O\left(\delta(s)+\frac{1}{\kappa(s)}\right) &\text{if $\delta(s)\to 0^+$ and $\kappa(s)\to +\infty$,}\\
		O(\delta(s)) &\text{if $\delta(s)\to 0^+$ and $\kappa(s){\in[{1}/{C'}, C'  ]}$,}\\
		O\left(\frac{1}{\abs{t}\sigma(s)}+\kappa(s)\right) &\text{if $\delta(s)\to 0^+$ and $\kappa(s)\to 0^+$}
		\end{cases}
		\]
		for every $C'>1$.
		Since $\delta(s)\leq \delta$ for every $s\in\R$, we may find some positive constants $C''$, $\delta_2\leq \delta_1$, where $\delta_1$ is that of Lemma~\ref{stime:lem:6}, and $\kappa_2\leq \kappa_1$ such that
		\[
		\abs{g(\abs{x}, \abs{t}\sigma(s))}\leq \begin{cases}
		C''\left(\delta(s)+\frac{1}{\kappa(s)}\right) &\text{when $\delta \leq \delta_2,\; \kappa(s)\geq \kappa_1$,}\\
		C''\delta(s) &\text{when $\delta \leq \delta_2,\; \kappa_2\leq\kappa(s)\leq \kappa_1$,}\\
		C''\left(\frac{1}{\abs{t}\sigma(s)}+\kappa(s)\right) &\text{when $\delta\leq \delta_2,\; \kappa(s)\leq  \kappa_2$.}
		\end{cases}
		\]
		We shall split the integrals accordingly. Notice first that we may assume also that {$\kappa_2\leq {1}/{(2 C)}\leq 2C\leq \kappa_1$}, and, up to taking a smaller $\delta_2$, that
		\[
		\phi(s)-2\pi\delta\sqrt{\sigma(s)}\rho\left(\delta(s) \right)\geq \frac{1}{2}\abs{s}
		\]
		whenever $\abs{s}\geq 2$ and $\delta\in [0,\delta_2)$.
		
		Consider first case {\bf III}, where {$\kappa \in [1/C,C]$}. We split
		
		\[\mathscr{I}'_\beta = \int_{\kappa(s) \leq \kappa_1} + \int_{\kappa(s) \geq \kappa_1} = \mathscr{I}'_{\beta,1} + \mathscr{I}'_{\beta,2}.  \]
		Observe that $\kappa(s) \geq \kappa_1$ if and only if $|s|\geq \sqrt{\frac{\kappa_1^4}{\kappa^4}-1 } \eqqcolon s_{1,\kappa}\geq 2$. Since
		\[\tilde I_{n+k_1-1}(\kappa(s)\rho(\delta(s)) e^{-\abs{t}\phi(s)}=O\left(e^{\abs{t}[2\pi\delta \sqrt{\sigma(s)}\rho(\delta(s))-\phi(s)   ] }\right)=O\left(e^{-\frac{1}{2}\abs{t}\abs{s}}\right)\]
		as $s\to \infty$, and since $\delta=O\left(\frac{1}{\kappa}\right)=O(1)$ in case {\bf III}, we get
		\begin{equation*}
			\begin{split}
				\abs{\mathscr{I}'_{\beta,2}} &\leq C' \left(\delta+\frac{1}{\kappa}\right)\int_{\substack{\abs{s} \geq s_{1,\kappa}} }   \sigma(s)^{\beta-\frac{1}{2}}\tilde I_{n+k_1-1}\left(\kappa(s) \rho\left(\delta(s)\right) \right) e^{-\abs{t}\phi(s) }  \,\dd s =O\left( e^{-\frac{s_{1,\kappa}}{4} \abs{t} } \right),
			\end{split}
		\end{equation*}
		which is negligible relative to $\frac{1}{\abs{t}^{\sfrac{3}{2}}}$. {By Laplace's method, moreover,}
		\begin{equation*}
			\abs{\mathscr{I}'_{\beta,1}}\leq C'\delta\int_{|s|\leq s_{1,\kappa} }   \sigma(s)^{\beta-\frac{1}{2}} \tilde I_{n+k_1-1}\left(\kappa(s)\rho\left(\delta(s) \right) \right) e^{-\abs{t}\phi(s) }  \,\dd s =O\left(  \delta\frac{1}{\sqrt{\abs{t}}}  \right)  
		\end{equation*}
		{with the same arguments as above.} {This concludes the study of case {\bf III}.}
		
		Consider now case {\bf IV}, where $\kappa\to 0^+$. We split 
		\[\mathscr{I}'_\beta = \int_{\kappa(s)\leq \kappa_2} +  \int_{\kappa_2 \leq \kappa(s) \leq \kappa_1} +\int_{\kappa(s) \geq \kappa_1} = \mathscr{I}'_{\beta,1} +\mathscr{I}'_{\beta,2} + \mathscr{I}'_{\beta,3}.  \]	
		Observe that $\kappa(s)\geq \kappa_2$ if and only if $s\geq \sqrt{\frac{\kappa_2^4}{\kappa^4}-1 }\eqqcolon s_{2,\kappa}$, and $s_{1,\kappa}\geq s_{2,\kappa}\geq 2$ if $\kappa$ is sufficiently small.	Exactly as above, we get
		\begin{equation*}
			\abs{\mathscr{I}'_{\beta,3}} \leq C' \left(\delta+\frac{1}{\kappa}\right)\int_{\substack{\abs{s} \geq s_{1,\kappa}} }   \sigma(s)^{\beta-\frac{1}{2}} \tilde I_{n+k_1-1}\left(\kappa(s)\rho\left(\delta(s) \right) \right) e^{-\abs{t}\phi(s) }  \,\dd s 
			=O\left(  \frac{1}{\kappa} e^{-\frac{s_{1,\kappa}}{4} \abs{t} } \right)  
		\end{equation*}
		which is negligible relative to $\frac{1}{\abs{t}^{\sfrac{3}{2}}}$. Then
		\begin{equation*}
			\abs{\mathscr{I}'_{\beta,2}} \leq C'\delta\int_{s_{2,\kappa}\leq\abs{s} \leq s_{1,\kappa}}   \sigma(s)^{\beta-\frac{1}{2}}\tilde I_{n+k_1-1}\left(\kappa(s) \rho\left(\delta(s) \right) \right) e^{-\abs{t}\phi(s) }  \,\dd s 
			=O\left(  \delta\, e^{-\frac{s_{2,\kappa}}{4} \abs{t} } \right)  ,
		\end{equation*}
		which is negligible relative to $\frac{1}{\abs{t}^{\sfrac{3}{2}}} $ in case \textbf{IV}. Finally,
		\[ \begin{split}
		\abs{\mathscr{I}'_{\beta,1}} &\leq C' \int_{\substack{ \abs{s} }\leq s_{2,\kappa}}   \sigma(s)^{\beta} \tilde I_{n+k_1-1}\left(\kappa(s)\rho\left(\delta(s) \right) \right) e^{-\abs{t}\phi(s) }\left(\sqrt{\sigma(s)}\kappa+\frac{1}{\sigma(s)\abs{t}} \right) \,\dd s
		\\&=O\left[ \frac{1}{\sqrt{\abs{t}}}\left(\frac{1}{\abs{t}}+\kappa \right) \right]  ,\end{split}
		\]
		{by Laplace's method as above.}
		The proof is complete.
	 \end{proof}
	We can finally state the following corollary, which is the natural extension of Corollary~\ref{cor:1}.
	\begin{corollary}\label{Hstime234}
		For $(x,t) \to \infty$ and $\delta \to 0^+$
		\[
		\begin{split} 
		p_{1,k_1,k_2}(x,t)= \frac{(-1)^{k_2} \pi^{k_1+k_2}}{2^{n-k_1+1 + \frac{m-1}{2}}}\abs{t}^{n+k_1-\frac{m+1}{2}}e^{-\frac{1}{4}d(x,t)^2 } e^{-\kappa \rho(\delta)}\tilde I_{n+k_1-1}(\kappa \rho(\delta)) \left[1+g(\abs{x},\abs{t})\right],
		\end{split}
		\]
		where
		\[
		g(\abs{x},\abs{t})=\begin{cases}
		O\left(\delta+\frac{1}{\kappa}\right) &\text{if $\delta\to 0^+$ and $\kappa\to +\infty$,}\\
		O(\delta) &\text{if $\delta\to 0^+$ and $\kappa \in{[1/C,C]}$,}\\
		O\left(\frac{1}{\abs{t}}+\kappa\right) &\text{if $\delta\to 0^+$ and $\kappa\to 0^+$}
		\end{cases}
		\]
		for every $C>0$.
	\end{corollary}

	We have not been able to find a single function which displays the asymptotic behaviour of $p_{1,k_1,k_2}(x,t)$ as $(x,t)\to \infty$, though we showed that the exponential decrease is the same in the four cases. This is also the same decrease found by Eldredge~\cite[Theorems 4.2 and 4.4]{Eldredge}, {when} $k_1=k_2=0$ and for the horizontal gradient, and Li~\cite[Theorems 1.4 and 1.5]{Li3}, {when} $k_1=k_2=0$. Notice that in~\cite[Theorem 1.5 and the following Remark (1)]{Li3} the remainders for $k_1=k_2=0$ seem to be better than the one we put in Corollary~\ref{Hstime234}, {but} {they} reduce to ours when developing the estimates in a more convenient form in cases {\bf II} and {\bf IV}, as we did in Theorems~\ref{HstimeII} and~\ref{HstimeIIIeIV}.
	
\begin{remark}\label{rem:sharpness}
Our sharp estimates for $p_{1,k_1,k_2}$ can be used to obtain asymptotic estimates of all the derivatives of the heat kernel $p_1$. Indeed, Faà di Bruno's formula leads to
\begin{multline}\label{partialdertx}
  \frac{\partial^{\abs{\gamma}}}{\partial x^{\gamma_1} \partial t^{\gamma_2}} p_{1}(x,t)\\ =\gamma_1 ! \gamma_2 ! \sum_{\eta,\mu,\beta} \frac{\abs{\mu}!2^{\abs{\mu_1}-\abs{\gamma_1}} }{\eta !\mu ! \beta !} \left[\prod_{h=1}^{\abs{\mu}} \left( \frac{\left(\frac{1}{2}\right)_h    }{h !}  \right)^{\beta_h}\right]x^{\eta_1} \sign(t)^{\mu_1} \abs{t}^{ \abs{\beta}-\abs{\gamma_2} } p_{1,\abs{\eta},\abs{\beta}}(x,t),
\end{multline}
where the sum is extended to all $\eta=(\eta_1,\eta_2)\in \N^{2n}\times \N^{2n}$, $\mu=(\mu_1,\mu_2)\in \N^m\times \N^m$ and $\beta \in \N^{\abs{\mu}}$ such that
\[
\gamma_1=\eta_1+2 \eta_2, \qquad \gamma_2=\mu_1+2\mu_2, \qquad \sum_{h=1}^{\abs{\mu}} h \beta_h= \abs{\mu}.
\]
However, the sharp asymptotic expansions we explicitly provided in Theorems~\ref{HstimeI},~\ref{HstimeII} and~\ref{HstimeIIIeIV} may not be enough to get {directly} \emph{sharp} asymptotic estimates of any desired derivative of $p_1$: some cancellations among the principal terms may indeed occur in~\eqref{partialdertx}. Nevertheless, by inspecting case by case, the interested reader could consider as many terms of the expansions given by Theorem~\ref{stationaryphase} {or Lemma~\ref{stime:lem:13}} as necessary. In the case when $t\to 0$, one may also make use of Lemma~\ref{lemmaTomega0} before expanding each term: a suitable choice for $N$ gets rid of the negative powers of $\abs{t}$ appearing in~\eqref{partialdertx}. Despite this, our estimates for $p_{1,k_1,k_2}$ lead to the \emph{sharp} behaviour at infinity of $\nabla_\Hc p_s$ and $\mathcal{L}p_s$, as we shall see in the next section.
\end{remark}	
	
	\section{Sub-Riemannian Ornstein Uhlenbeck Operators}\label{OU}
	For every $s>0$ consider the operator on $L^2(p_s)$ given by
	\begin{equation*}
		\mathcal{L}^{p_s}= \mathcal{L} - \frac{\nabla_\mathcal{H} {p_s}}{{p_s}} \cdot \nabla_\Hc {\colon C_c^\infty \to L^2(p_s)}
	\end{equation*}
	which arises from the Dirichlet form $\phi \mapsto \int_{G} \abs{\nabla_\mathcal{H}\phi(y)} ^2 p_s(y)\, \dd y$. For a fixed time $s>0$, $\mathcal{L}^{p_s}$ can be considered as a sub-Riemannian version of the classical Ornstein-Uhlenbeck operator (see~\cite{Baudoinetal, LustPiquard}). Arguing as Strichartz (\cite[Theorem 2.4]{Strichartz}) it is not hard to see that $\mathcal{L}^{p_s}$ with domain $C_c^\infty(G)$ is essentially self-adjoint on $L^2(p_s)$, for every $s>0$. Let us then consider its closure, which we still denote by $\mathcal{L}^{p_s}$.
	
	\begin{theorem}\label{teo:discrspectr}
		$\mathcal{L}^{p_s}$ has purely discrete spectrum for all $s>0$.
	\end{theorem}
	
	Theorem~\ref{teo:discrspectr} is indeed due to Inglis~\cite{Inglis}, whose proof relies on super Poincaré inequalities. Instead, we reduce {matters to studying} a Schr\"odinger-type operator by conjugating $\mathcal{L}^{p_s}$ with the isometry $U_s\colon L^2(p_s) \rightarrow L^2$ defined by $U_s f=f\sqrt{p_s}$ (see e.g.~\cite{MetafunePallara, Carbonaro}). More precisely, we consider the operator $U_s\, \mathcal{L}^{p_s} \, U_s^{-1} \colon L^2 \rightarrow L^2$. Simple computations then lead to $U_s \mathcal{L}^{p_s}\, U_s^{-1}=\mathcal{L}+V_s$, where $V_s$ is the multiplication operator\footnote{With a slight abuse of notation, we do not distinguish between a multiplication operator by a function and the function itself.} given by the function 
	\[
	V_s= -\frac{1}{4}\frac{\abs*{\nabla_\mathcal{H}p_s}^2}{p_s^2} - \frac{1}{2}\frac{\mathcal{L}p_s}{p_s}=-\frac{1}{4}\frac{\sum_{j=1}^{2n} (X_j p_s)^2 }{p_s^2} { + \frac{1}{2}\frac{\sum_{j=1}^{2n} X_j^2 p_s}{p_s}.}
	\]
	The main ingredient of the proof is due to Simon~\cite[Theorem 2]{Simon1}. Given a potential $V$ and $M>0$, we define $\Omega_M \coloneqq \{ g\in G \colon \, V(g) \leq M \}$. {For a subset $E$ of $G$, we write $\abs{E}$ to denote its measure with respect to $\dd y$.}
	
	\begin{theorem}\label{Simon1}
		Let $V$ be a potential bounded from below such that $\abs*{\Omega_{M}}<\infty$ for every $M>0$. Then there exists a self-adjoint extension of $\mathcal{L} +V$ with purely discrete spectrum.
	\end{theorem}
	In order to apply Proposition~\ref{Simon1}, some estimates of the potential are needed; this is done in the following proposition. 
	\begin{proposition} \label{stime:teo:2}
		When $(x,t)\rightarrow\infty$, $V_s(x,t)\asymp s^{-2}d(x,t)^2$ for every $s>0$.
	\end{proposition}

	\begin{proof}
		Since $V_s(x,t)=\frac{1}{s}V_1\left( \frac{x}{\sqrt s},\frac{t}{s}\right)$, it will be sufficient to consider $V_1$ only. For every $(x,t)\in G$
		\begin{equation}\label{gradient}
			\begin{split}
				\abs{\nabla_\Hc p_1}^2(x,t)= R\, p_{1,1,0}(x,t)^2+R\, p_{1,0,1}(x,t)^2,
			\end{split}
		\end{equation}
		while
		\[
		\begin{split}
		\mathcal{L} p_1(x,t)&=-R\, p_{1,2,0}(x,t)-n \,p_{1,1,0}(x,t)- R\, p_{1,0,2}(x,t) { + \frac{R}{\abs{t}}(m-1)p_{1,0,1}(x,t)}.
		\end{split}
		\]
		Hence
		\[
		V_1=-\frac{R}{4} \frac{  p_{1,1,0}^2+p_{1,0,1}^2 }{p_{1,0,0}^2}+\frac{R}{2}\frac{ p_{1,2,0}+p_{1,0,2}+\frac{n}{R} {p_{1,1,0}} - \frac{m-1}{\abs{t}} p_{1,0,1} }{p_{1,0,0}}.
		\]
		In order to find the asymptotics for the potential, it turns out that only the principal term of $p_{1,k_1,k_2}$ is necessary, and therefore, for the sake of simplicity, we shall avoid an explicit treatment of the remainders. If one is interested in a more detailed description of the behaviour of the potential, however, it is enough to use the remainders that we found in the previous sections. 
		\smallskip
		
		{\bf I.} If $\omega$ runs through $[0,C]$ for some $C>0$, then {both} $ \frac{y_\omega}{\sin(y_\omega)} $ {and $\frac{y_\omega}{\omega}$ are positive and} bounded both from above and from below.  Hence,
		
		\[
		\begin{split}
		V_1 (x,t) &\sim-\frac{R}{4}\frac{y_\omega^{2} }{\sin(y_\omega)^{2}}+\frac{R}{2}\frac{y_\omega^{2} }{\sin(y_\omega)^{2}} =\frac{R}{4}\frac{y_\omega^{2} }{\sin(y_\omega)^{2}}\asymp d(x,t)^2
		\end{split}
		\]
		thanks to Theorem~\ref{HstimeI}.
		
		{\bf  II.} {Let} $\delta\to 0^+$ and $\kappa\to +\infty$. Then $\frac{1}{R\delta}=o\left(\frac{1}{\delta^2\sqrt \kappa} \right) $, and $\kappa+\frac{\abs*{t}}{\sqrt \kappa}=o(\abs*{t})$. Therefore, by Theorem~\ref{HstimeII},
		\[
		\begin{split}
		V_1 (x,t) 
		&\sim-\frac{R}{4}\frac{1}{\delta^2}+\frac{R}{2}\frac{1}{\delta^2} =\frac{\pi}{4}\abs*{t}\asymp  d(x,t)^2.
		\end{split}
		\]
		
		{\bf  III.} {Let} $\delta\to 0^+$ and {$\kappa \in [1/C,C]$} for some $C>1$. Then $\delta\asymp R$. Elementary computations yield
		\[
		I'_{\nu}(\zeta )=\frac{I_{\nu-1}(\zeta )+I_{\nu+1}(\zeta )}{2},\]
		{so that}
		\[
		(2 I_{\nu-1}I_{\nu+1}-I_\nu^2)'(\zeta )=I_{\nu-2}(\zeta )I_{\nu+1}(\zeta )+I_{\nu-1}(\zeta )I_{\nu+2}(\zeta )
		\]	
		for all $\nu\in \Z$ and for all $\zeta \in \C$. Thus, $2 I_{n-1}I_{n+1}-I_n^2$ is strictly increasing on $[0,\infty)$, hence strictly positive on $(0,\infty)$. Therefore, by Theorem~\ref{HstimeIIIeIV}
		\[
		\begin{split}
		V_1 (x,t)	&\sim-\frac{R}{4}\frac{\frac{1}{\delta^2}I_{n}(\kappa)^2
		}{I_{n-1}(\kappa)^2}+\frac{R}{2}\frac{\frac{1}{\delta^2}I_{n+1}(\kappa)
		+\frac{n}{R\delta} I_{n}(\kappa) 	
	}{I_{n-1}(\kappa)		}\\
	&=\frac{\pi\abs*{t}}{4}\frac{\frac{2 n}{\kappa}I_n(\kappa)I_{n-1}(\kappa)+  2 I_{n-1}(\kappa)I_{n+1}(\kappa)-I_n(\kappa)^2}{I_{n-1}(\kappa)^2}\asymp  d(x,t)^2.
	\end{split}
	\]
	
	{\bf IV.} Finally, {let} $\kappa\to 0^+$ and $ \abs{t}\to +\infty$. Then $\abs{t}=o\left(\frac{1}{R}\right)$, so that 
	\[
	\begin{split}
	V_1(x,t)	&\sim-\frac{R}{4} \frac{\frac{\pi^2}{(n!) ^2}\abs*{t}^{2}+\frac{\pi^2}{[(n-1)!] ^2}
	}{\frac{1}{[(n-1)!]^2}}+ \frac{R}{2}\frac{\frac{\pi^2}{(n+1)!}\abs*{t}^{2} +\frac{\pi^2}{(n-1)!}+ \frac{n\pi}{R\, n!} \abs*{t}+ \frac{(m-1)\pi}{(n-1)!\abs{t}}}{\frac{1}{(n-1)!}}\\
	&\sim\frac{\pi}{2}\abs*{t} \asymp d(x,t)^2,
	\end{split}
	\]		
	thanks to Theorem~\ref{HstimeIIIeIV} again.
 \end{proof}

\begin{remark}\label{remarknotenough}
	The estimates provided by Eldredge~\cite{Eldredge} are not sufficient to prove Proposition~\ref{stime:teo:2}, not even with some precise estimates of $\mathcal{L}p_1/p_1$. Indeed, as the proof above shows, in cases {\bf I}, {\bf II} and {\bf III} one has $\mathcal{L} p_1/p_1\asymp \abs{\nabla_\Hc p_1}^2/p^2_1$, so that no lower control of $V_1$ can be inferred. On the other hand, the upper bounds of the derivatives of $p_s$ explicitly provided by Li~\cite{Li3} are not enough to describe the behaviour at infinity of $V_s$.
\end{remark}

\begin{proof}[Proof of Theorem~\ref{teo:discrspectr}]
	Since $V_s$ is continuous and diverges at infinity by Proposition~\ref{stime:teo:2}, {the assumptions of Theorem~\ref{Simon1} are fulfilled and this ensures} the existence of a self-adjoint extension $(T_s,\mathscr{D}_s)$ of $(\mathcal{L}+V_s,C_c^\infty)$ with purely discrete spectrum.  Since the multiplication by the square root of $p_s$, which we called $U_s$, preserves $C_c^\infty$, $U_s^{-1}\mathscr{D}_s \supseteq C_c^\infty$; therefore, $(U_s^{-1} T_s U_s, U_s^{-1}\mathscr{D}_s)$ is a self-adjoint extension -- with purely discrete spectrum -- of $(\mathcal{L}^{p_s}, C_c^\infty)$, which is essentially self-adjoint. The result follows.
 \end{proof}

\section*{Acknowledgements}
We would like to thank Giancarlo Mauceri for his support and several helpful suggestions and encouraging discussions. We also thank the anonymous referee for his/her valuable comments on a preliminary version of this paper, which led us to state Theorem~\ref{HstimeI} in its present form.

\end{document}